\documentclass[11pt]{amsart}
\usepackage{geometry}               
\usepackage{amssymb,amsmath,latexsym}
\usepackage[english]{babel}
\usepackage{graphicx}
\usepackage[latin1]{inputenc}
\usepackage{float}
\usepackage[caption = false]{subfig}
\usepackage{xypic}
\usepackage{multirow}
\usepackage{verbatim}
\usepackage{tikz-cd}
\usetikzlibrary{cd}

\newtheorem{theorem}{Theorem}[section]
\newtheorem{lemma}[theorem]{Lemma}
\newtheorem{corollary}[theorem]{Corollary}
\newtheorem{proposition}[theorem]{Proposition}
\newtheorem{definition}[theorem]{Definition}

\usepackage[all]{xy}
\numberwithin{equation}{section}

\title{Fenchel-Nielsen coordinates for ${\rm SL}(3,{\mathbb C})$ representations}
\author{Rodrigo D\'avila Figueroa \& John R.\,Parker}
\date{\today}                                           

\begin{document}
\maketitle

\begin{abstract}
We define Fenchel-Nielsen coordinates for representations of surface groups to ${\rm SL}(3,{\mathbb C})$. We also show how
these coordinates relate to the classical Fenchel-Nielsen coordinates and to their generalisations by Kourouniotis, Tan, Goldman,
Zhang and Parker-Platis. 
\end{abstract}

\section{Introduction} 

The Teichm\"uller space of a closed, orientable Riemann surface $S_g$ of genus $g\geq2$ is the space of marked 
hyperbolic structures on $S_g$ 
up to isotopy. Fenchel and Nielsen constructed global coordinates on this space. The coordinates depend on a choice 
${\mathcal L}$ of $3g-3$ distinct, non-trivial homotopy classes of simple closed curves on the surface with disjoint 
representatives.
The coordinates are the set 
of hyperbolic lengths of geodesics in each homotopy class in ${\mathcal L}$, together with twists around these geodesics. 
An alternative description of the Teichm\"uller space of $S_g$ is the space of irreducible, discrete, totally loxodromic 
representations of $\pi_1(S_g)$ to 
${\rm PSL}(2,{\mathbb R})={\rm Isom}_+(S_g)$ up to conjugation. In this context, Fenchel-Nielsen length coordinates are 
equivalent to 
the $3g-3$ traces of the elements of ${\rm PSL}(2,{\mathbb R})$ representing the curves in ${\mathcal L}$ and the twist
coordinates are traces (or eigenvalues) of elements in their centralisers. 

The second definition of Techm\"uller space can be generalised to representations of $\pi_1(S_g)$ to other groups $G$ and
so Fenchel-Nielsen coordinates can be generalised as well. Particular cases where Fenchel-Nielsen coordinates have been 
constructed are when $G$ is one of ${\rm SL}(2,{\mathbb C})$, Kourouniotis \cite{Kou}, Tan \cite{Tan}; 
${\rm PSL}(3,{\mathbb R})$, Goldman \cite{Gol} or ${\rm SU}(2,1)$, Parker and Platis \cite{PP}. In the case of 
${\rm SL}(2,{\mathbb C})$ and ${\rm SU}(2,1)$ there is an important change from the case of ${\rm SL}(2,{\mathbb R})$, namely
the space of totally loxodromic representations is a proper subset of the corresponding compoenent of the representation variety.
Indeed, in the ${\rm SL}(2,{\mathbb C})$ case, the space of discrete, faithful, totally loxodromic representations has 
fractal boundary, and this boundary can intersect itself in complicated ways \cite{AndCan}. 
Therefore the allowable values of these coordinates describe a larger space containing the space of discrete, faithful,
totally loxodromic representations as a proper subspace. 

In this paper we construct Fenchel-Nielsen coordinates to the case where $G={\rm SL}(3,{\mathbb C})$. 
Both ${\rm SL}(3,{\mathbb R})$ and ${\rm SU}(2,1)$ are subgroups of ${\rm SL}(3,{\mathbb C})$ and, after taking the 
irreducible representation, we can embed ${\rm SL}(2,{\mathbb R})$ and ${\rm SL}(2,{\mathbb C})$ as subgroups of 
${\rm SL}(3,{\mathbb C})$ 
as well. We show how our coordinates are related to those constructed in \cite{Kou,Tan,Gol,PP}. Our motivation for considering
${\rm SL}(3,{\mathbb C})$ representations of surface groups is the study of complex Kleinian groups, which are discrete 
subgroups 
of ${\rm SL}(3,{\mathbb C})$; see the book \cite{C-N-S} by Cano, Navarrete, Seade. Much of the focus of \cite{C-N-S} and 
related papers has been
the case of elementary or Fuchsian groups. We hope that defining Fenchel-Nielsen coordinates for 
${\rm SL}(3,{\mathbb C})$ representations of surface groups will facilitate the study of their action on 
${\mathbb C}{\mathbb P}^2$, as studied in \cite{C-N-S}. 

\noindent
{\bf Acknowledgements.}
This paper is part of the PhD thesis of the first author. Both authors are grateful to Sean Lawton for helpful conversations.

\section{Statement of main results}

The Teichm\"uller space of a surface $S_g$ is defined as follows.

\begin{definition}\label{def-Teich}
	Let $S_g$ be a closed, compact surface of genus $g\geq2$. The Teichm\"uller space of $S_g$ is the quotient 
	$\mathcal{T}(S_g)=\{(X,f)\}/\sim$, where
	\begin{itemize}
		\item $X$ is $S_g$ together with a hyperbolic structure.
		\item $f:S_g\rightarrow X$ is a homeomorphism called a marking.
		\item $(X,f)\sim(Y,g)$ if and only if there exists an isometry $\phi:X\rightarrow Y$ such that $\phi\circ f$ is isotopic to $g$.
	\end{itemize}
\end{definition}

Our results depend on a choice of a curve system on $S_g$.

\begin{definition}\label{def-curve-system}
Let $S_g$ be a closed, oriented surface of genus $g\ge 2$. A {\sl curve system} on $S_g$ denoted
$\mathcal{L}=\{[\gamma_1],\,\ldots,\,[\gamma_{3g-3}]\}$ is a maximal set of distinct, non-trivial homotopy 
classes of simple closed curves in $S_g$ with simple, disjoint representatives $\gamma_j$.
The set $S_g\backslash\cup_{j=1}^{3g-3}\gamma_j$ is a decomposition of the 
surface into $2g-2$ pairs of pants (three holed spheres) $Y_1,\ldots,Y_{2g-2}$. Each curve 
$\gamma_j$ is the boundary of precisely two pairs of pants (including the case when a curve corresponds two different 
boundary curves of the same pair of pants).

When we consider $X=f(S_g)$, that is $S_g$ together with a hyperbolic 
structure, we always assume that $f$ sends $\gamma_j$ to the geodesic on $X$ in the homotopy class $[\gamma_j]$. 
By a mild abuse of notation, we call this image $\gamma_j$ as well. 
\end{definition}

%

We consider representations to ${\rm SL}(3,{\mathbb C})$ where $A_j=\rho([\gamma_j])$ is strongly loxodromic, that
is the eigenvalues of $A_j$ have pairwise different absolute values, see Theorem~\ref{thm-SL3C-class} below.
We are now in a position to give our main definition, which should be thought of as an extension to
${\rm SL}(3,{\mathbb C})$ of the classical definition of quasi-Fuchsian representations of surface groups.

\begin{definition}\label{def-FN-space}
Let $S_g$ be a closed, oriented surface of genus $g\ge 2$ and let ${\mathcal L}$ be a curve system on $S_g$,
Let $[\gamma_j]$ for $j=1,\,\ldots,\, 3g-3$ and $Y_k$ for $k=1,\,\ldots,\,2g-2$ be as above.
Let ${\mathcal Y}(S_g,{\mathcal L})$ be the space of conjugacy classes of representations
$\rho:\pi_1(S_g)\longrightarrow {\rm SL}(3,{\mathbb C})$ so that
\begin{enumerate}
\item[(1)] for each $j=1,\,\ldots,\,3g-3$ the curve $\gamma_j$ is represented by a strongly loxodromic map 
$A_j=\rho([\gamma_j])$; and
\item[(2)] for each $k=1,\,\ldots,\,2g-2$ the restriction of $\rho$ to $\pi_1(Y_k)$ is irreducible.
\end{enumerate}
\end{definition}

Our first result gives coordinates on the representations in ${\mathcal Y}(S_g,{\mathcal L})$.
The relevant definitions may be found later in the paper, as follows: shape invariants in
Definition~\ref{def-shape-invariants}; commutator equations \eqref{eqn:Q-y} from 
Theorem~\ref{thm-lawton-new}; twist-bend and bulge-turn  parameters in Definition~\ref{def-t-b-b-t}.

\begin{theorem}\label{thm:main-SL(3,C)}
Let $S_g$ and ${\mathcal L}$ be as above. Let $\rho\in{\mathcal Y}(S_g,{\mathcal L})$ and write
$\Gamma=\rho\bigl(\pi_1(S_g)\bigr)$.
For $j=1,\,\ldots,\,3g-3$ write $A_j=\rho([\gamma_j])$, write $(t+i\theta)_\Gamma(\gamma_j)$ 
for the complex twist-bend along $\gamma_j$ and write $(s+i\phi)_\Gamma(\gamma_j)$ 
for the complex bulge-turn along $\gamma_j$. For $k=1,\,\ldots,\,2g-2$ write 
$\sigma_+(Y_k)$ and $\sigma_-(Y_k)$ for the shape invariants of $Y_k$. 
Then $\Gamma$ is determined up to conjugation by
\begin{enumerate}
\item[(1)] the traces ${\rm tr}(A_1),\,\ldots,\,{\rm tr}(A_{3g-3})$ and ${\rm tr}(A_1^{-1}),\,\ldots,\,{\rm tr}(A_{3g-3}^{-1})$;
\item[(2)] the shape invariants $\sigma_+(Y_1),\,\ldots,\, \sigma_+(Y_{2g-2})$ and 
$\sigma_-(Y_1),\,\ldots,\, \sigma_-(Y_{2g-2})$; 
\item[(3)] the choice of a root of the commutator equations $Q(Y_1),\,\ldots,\,Q(Y_{2g-2})$; 
\item[(4)] the twist-bend parameters $(t+i\theta)_\Gamma(\gamma_1),\,\ldots,\,(t+i\theta)_\Gamma(\gamma_{3g-3})$
and the bulge-turn parameters $(s+i\phi)_\Gamma(\gamma_1),\,\ldots,\,(s+i\phi)_\Gamma(\gamma_{3g-3})$.
\end{enumerate}
\end{theorem}

Our main interest is in the space of conjugacy classes of discrete, faithful, totally strongly loxodromic representations
of $\pi_1(S_g)$ to ${\rm SL}(3,{\mathbb C})$, or to Lie subgroups $G$ of ${\rm SL}(3,{\mathbb C})$. Clearly, such
representations form a proper subset of ${\mathcal Y}(S_g,{\mathcal L})$. There are several reasons for this. 
The first reason is that, because in that definition, we have only assumed $3g-3$ conjugacy classes (and hence their powers) 
are strongly loxodromic. It may be that other elements in the group are represented by non-loxodromic maps. Secondly,
even if the group is totally strongly loxodromic, it may be non-discrete. To see this, we observe that in 
${\rm SL}(2,{\mathbb C})$, and hence in its image in ${\rm SL}(3,{\mathbb C})$ under the irreducible representation,
there are non-discrete totally (strongly) loxodromic representations of surface groups. 

Furthermore, we chose to consider representations to ${\rm SL}(3,{\mathbb C})$, or to a Lie subgroup, 
that may be continuously deformed through discrete, faithful, totally strongly loxodromic representations from the space of
Fuchsian representations to ${\rm SO}_0(2,1)$. We call such representations quasi-Fuchsian. 
It would also be possible to consider other components
of the space of discrete, faithful representations of $\pi_1(S_g)$, see for example Goldman, Kapovich, Leeb \cite{G-K-L},
but in this case, care would be needed when considering ${\mathbb C}$-Fuchsian pairs of pants. 

\begin{definition}\label{def-Fuchs-space}
The Fuchsian space, denoted ${\mathcal F}(S_g)$, for the surface $S_g$
is the subspace of ${\rm Hom}\bigl(\pi_1(S_g),{\rm SO}_0(2,1)\bigr)//{\rm SO}_0(2,1)$ of irreducible, discrete, faithful, 
totally loxodromic representations $\rho:\pi_1(S_g)\longrightarrow {\rm SO}_0(2,1)$ up to conjugation. 
\end{definition}

We remark that ${\rm SO}_0(2,1)$ and 
${\rm PSL}(2,{\mathbb R})$ are isomorphic (for a concrete isomorphism, see Section~\ref{sec:SL(2,K)} below), 
and both are the orientation preserving isometry groups of the hyperbolic plane.
Note that ${\mathcal F}(S_g)$ is a copy of Teichm\"uller space ${\mathcal T}(S_g)$ from Definition~\ref{def-Teich}.
We now give an ${\rm SO}_0(2,1)$ version of the theorem of Fenchel and Nielsen, see Section~\ref{sec:alg-FN} below.

\begin{theorem}[Fenchel-Nielsen]\label{thm:main-SO(2,1)}
Let $S_g$ and ${\mathcal L}$ be as above. Let $\rho\in{\mathcal F}(S_g)$. 
For $j=1,\,\ldots,\,3g-3$ write $A_j=\rho([\gamma_j])$ and write $t_\Gamma(\gamma_j)$ for the twist along $\gamma_j$. 
Then $\Gamma$ is determined up to conjugation by
\begin{enumerate}
\item[(1)] the traces ${\rm tr}(A_1),\,\ldots,\,{\rm tr}(A_{3g-3})$, where each trace lies in $[3,\infty)$,
\item[(2)] the twists $t_\Gamma(\gamma_1),\,\ldots,\,t_\Gamma(\gamma_{3g-3})$, where each twist lies in ${\mathbb R}$. 
\end{enumerate}
\end{theorem}

In order to relate this result to Theorem~\ref{thm:main-SL(3,C)} we observe that for such representations:
the traces are real and greater than $3$ and satisfy ${\rm tr}(A_j)={\rm tr}(A_j^{-1})$; the 
bend, bulge and turn parameters $\theta_\Gamma(\gamma_j)$, $s_\Gamma(\gamma_j)$, $\phi_\Gamma(\gamma_j)$
are all zero, the commutator equations have repeated roots and the shape invariants satisfy
\begin{eqnarray}
\sigma_+(Y_k) = \sigma_-(Y_k)  & = & {\rm tr}(A_{k_1})+{\rm tr}(A_{k_2})+{\rm tr}(A_{k_3})+1 \label{eq:shape-SL} \\
&& \quad +2\sqrt{\bigl({\rm tr}(A_{k_1})+1\bigr)\bigl({\rm tr}(A_{k_2})+1\bigr)\bigl({\rm tr}(A_{k_3})+1\bigr)}, \notag
\end{eqnarray}

We can now define quasi-Fuchsian space for ${\rm SL}(3,{\mathbb C})$ and its Lie subgroups.

\begin{definition}
Let $G$ be a group satisfying ${\rm SO}_0(2,1)<G<{\rm SL}(3,{\mathbb C})$. Let $S_g$ be a closed, oriented
surface of genus $g\ge 2$ and let ${\mathcal L}$ be a curve system on $S_g$. Define 
the $G$-quasi-Fuchsian space of $S_g$ with respect to ${\mathcal L}$, written ${\mathcal Q}(S_g,{\mathcal L},G)$, 
to be the collection of ${\rm SL}(3,{\mathbb C})$-conjugacy 
classes of representations $\rho:\pi_1(S_g)\longrightarrow G$ satisfying
\begin{enumerate}
\item[(1)] each representation in ${\mathcal Q}(S_g,{\mathcal L},G)$ is discrete, faithful and totally strongly loxodromic.
\item[(2)] ${\mathcal Q}(S_g,{\mathcal L},G)$ contains the Fuchsian space ${\mathcal F}(S_g)$. That is, 
the ${\rm SL}(3,{\mathbb C})$-conjugacy class of each Fuchsian representation 
$\rho_0:\pi_1(S_g)\longrightarrow {\rm SO}_0(2,1)<G$ lies in ${\mathcal Q}(S_g,{\mathcal L},G)$. 
\item[(3)] ${\mathcal Q}(S_g,{\mathcal L},G)$ is path connected. That is, each 
representation $\rho\in {\mathcal Q}(S_g,{\mathcal L},G)$ may be connected to a Fuchsian representation
by a path of representations lying in ${\mathcal Q}(S_g,{\mathcal L},G)$.
\end{enumerate}
\end{definition}
%
%
%

We remark that in the case of $G={\rm SO}(3,{\mathbb C})$ and ${\rm SU}(2,1)$ the space 
${\mathcal Q}(S_g,{\mathcal L},G)$ is not a component of the representation variety. In these cases, we may move
towards the boundary of the space of discrete and faithful representations by pinching a simple closed curve on $S_g$ so that
in the limit is represented by a parabolic map and then beyond to representations where it becomes elliptic of infinite order,
and so the representation ceases to be discrete.


Moreover,  the requirement (3) means that a representation in ${\mathcal Q}(S_g,{\mathcal L}, G)$ should be connected
by a path of representations in ${\mathcal Q}(S_g,{\mathcal L},G)$ to a Fuchsian representation in
${\mathcal F}(S_g)$.
This is more restrictive than simply requiring a representation whose image lies in $G$. In particular, 
when $G={\rm SL}(3,{\mathbb R})$ then each loxodromic map must have positive eigenvalues. This is
because all eigenvalues of loxodromic maps in ${\rm SO}_0(2,1)$ are positive, and when continuously deforming
through loxodromic maps we cannot have eigenvalue $0$. Similarly, when $G={\rm SU}(2,1)$ we need to be
in the component of the deformation space containing ${\rm SO}_0(2,1)$ representations, and hence the
Toledo invariant must be zero.

%


Since ${\mathcal Q}\bigl(S_g,{\mathcal L},{\rm SL}(3,{\mathbb C})\bigr)$ is a subset of ${\mathcal Y}(S_g,{\mathcal L})$
we can repeat Theorem~\ref{thm:main-SL(3,C)} word for word. 
We are interested in ${\mathcal Q}(S_g,{\mathcal L},G)$ in the cases where 
$G$ is one of ${\rm SO}(3;{\mathbb C})$ (that is the irreducible representation of 
${\rm PSL}(2,{\mathbb C})$), ${\rm SL}(3,{\mathbb R})$ or ${\rm SU}(2,1)$. 
We now give Fenchel-Nielsen coordinates for ${\mathcal Q}(S_g,{\mathcal L},G)$ in each of these cases. 
In our observations following Theorem~\ref{thm:main-SO(2,1)} we observed that a representation being 
Fuchsian placed certain restrictions on the parameters from Theorem~\ref{thm:main-SL(3,C)}. Similar
restrictions apply for each of the other groups $G$. These restrictions cut out subspaces of 
${\mathcal Q}\bigl(S_g,{\mathcal L},{\rm SL}(3,{\mathbb C})\bigr)$ which are all linear, apart from the condition
\eqref{eq:shape-SL} for the shape invariant when $G={\rm SO}_0(2,1)$ or ${\rm SO}(3,{\mathbb C})$.


%
%

Our next result concerns the irreducible representation of ${\rm SL}(2,{\mathbb C})$ to ${\rm SL}(3,{\mathbb C})$. The
image of this representation is ${\rm SO}(3;{\mathbb C})$, see below. The next three results are re-interpretations
of the theorems of Kourouniotis \cite{Kou}, Tan \cite{Tan}, of Goldman \cite{Gol} and of Parker-Platis \cite{PP}
respectively.

\begin{theorem}\label{thm:main-SO(3,C)} 
Let $S_g$ and ${\mathcal L}$ be as above. Let $\rho\in{\mathcal Q}\bigl(S_g,{\mathcal L}, {\rm SO}(3;{\mathbb C})\bigr)$ and write 
$\Gamma=\rho\bigl(\pi_1(S_g)\bigr)$.
For $j=1,\,\ldots,\,3g-3$ write $A_j=\rho([\gamma_j])$, write $(t+i\theta)_\Gamma(\gamma_j)$ 
for the complex twist bend along $\gamma_j$. For $k=1,\,\ldots,\,2g-2$ write 
$\sigma_+(Y_k)$ and $\sigma_-(Y_k)$ for the shape invariants of $Y_k$. 
Then $\Gamma$ is determined up to conjugation by
\begin{enumerate}
\item[(1)] the traces ${\rm tr}(A_1),\,\ldots,\,{\rm tr}(A_{3g-3})$; 
\item[(2)] the twist-bends $(t+i\theta)_\Gamma(\gamma_1),\,\ldots,\,(t+i\theta)_\Gamma(\gamma_{3g-3})$.
\end{enumerate}
\end{theorem}

We remark that ${\mathcal Q}\bigl(S_g,{\mathcal L},{\rm SO}(3;{\mathbb C})\bigr)$ is contained in the subset of
${\mathcal Q}\bigl(S_g,{\mathcal L},{\rm SL}(3,{\mathbb C})\bigr)$ where
\begin{enumerate}
\item[(1)] ${\rm tr}(A_j)={\rm tr}(A_j^{-1})\in{\mathbb C}$ for $j=1,\,\ldots,\,3g-3$;
\item[(2)] if $\gamma_{k_1},\,\gamma_{k_2},\,\gamma_{k_3}$ are the boundary curves of $Y_k$ then the
shape invariants $\sigma_+(Y_k)=\sigma_-(Y_k)$ and satisfy equation \eqref{eq:shape-SL}.
\item[(3)] the commutator equations $Q_k$ all have repeated roots;
\item[(4)] the bulge-turn parameters $s_\Gamma(\gamma_j)$, $\phi_\Gamma(\gamma_j)$ are all zero.
\end{enumerate}

\begin{theorem}\label{thm:main-SL(3,R)}
Let $S_g$ and ${\mathcal L}$ be as above. Let $\rho\in{\mathcal Q}\bigl(S_g,{\mathcal L}, {\rm SL}(3,{\mathbb R})\bigr)$ and write 
$\Gamma=\rho\bigl(\pi_1(S_g)\bigr)$.
For $j=1,\,\ldots,\,3g-3$ write $A_j=\rho([\gamma_j])$, write $t_\Gamma(\gamma_j)$ 
for the twist bend $\gamma_j$ and write $s_\Gamma(\gamma_j)$ 
for the bulge along $\gamma_j$. For $k=1,\,\ldots,\,2g-2$ write 
$\sigma_+(Y_k)$ and $\sigma_-(Y_k)$ for the shape invariants of $Y_k$. 
Then $\Gamma$ is determined up to conjugation by
\begin{enumerate}
\item[(1)] the traces ${\rm tr}(A_1),\,\ldots,\,{\rm tr}(A_{3g-3})$ and ${\rm tr}(A_1^{-1}),\,\ldots,\,{\rm tr}(A_{3g-3}^{-1})$;
\item[(2)] the shape invariants $\sigma_+(Y_1),\,\ldots,\, \sigma_+(Y_{2g-2})$ and 
$\sigma_-(Y_1),\,\ldots,\, \sigma_-(Y_{2g-2})$; 
\item[(3)] the choice of a root of the commutator equations $Q(Y_1),\,\ldots,\,Q(Y_{2g-2})$; 
\item[(4)] the twists $t_\Gamma(\gamma_1),\,\ldots,\,t_\Gamma(\gamma_{3g-3})$
and the bulges $s_\Gamma(\gamma_1),\,\ldots,\,s_\Gamma(\gamma_{3g-3}$.
\end{enumerate}
\end{theorem}

We remark that ${\mathcal Q}\bigl(S_g,{\mathcal L},{\rm SL}(3,{\mathbb R})\bigr)$ is (contained in) the subset of
${\mathcal Q}\bigl(S_g,{\mathcal L},{\rm SL}(3,{\mathbb C})\bigr)$ where
\begin{enumerate}
\item[(1)] ${\rm tr}(A_j)$ and ${\rm tr}(A_j^{-1})\in{\mathbb R}_+$ for $j=1,\,\ldots,\,3g-3$;
\item[(2)] the shape invariants $\sigma_+(Y_k)$ and $\sigma_-(Y_k)$ are real;
\item[(3)] the bend and turn parameters $\theta_\Gamma(\gamma_j)$, $\phi_\Gamma(\gamma_j)$ are all zero.
\end{enumerate}

\begin{theorem}\label{thm:main-SU(2,1)}
Let $S_g$ and ${\mathcal L}$ be as above. Let $\rho\in{\mathcal Q}\bigl(S_g,{\mathcal L}, {\rm SU}(2,1)\bigr)$ and write 
$\Gamma=\rho\bigl(\pi_1(S_g)\bigr)$.
For $j=1,\,\ldots,\,3g-3$ write $A_j=\rho([\gamma_j])$, write $t_\Gamma(\gamma_j)$ 
for the twist along $\gamma_j$ and $\phi_\Gamma(\gamma_j)$ for the turn along $\gamma_j$.
For $k=1,\,\ldots,\,2g-2$ write $\sigma_+(Y_k)$ and $\sigma_-(Y_k)$ for the shape invariants of $Y_k$. 
Then $\Gamma$ is determined up to conjugation by
\begin{enumerate}
\item[(1)] the traces ${\rm tr}(A_1),\,\ldots,\,{\rm tr}(A_{3g-3})$; 
\item[(2)] the shape invariants $\sigma_+(Y_1),\,\ldots,\, \sigma_+(Y_{2g-2})$;
\item[(3)] the choice of a root of the commutator equations $Q(Y_1),\,\ldots,\,Q(Y_{2g-2})$; 
\item[(4)] the twists $t_\Gamma(\gamma_1),\,\ldots,\,t_\Gamma(\gamma_{3g-3})$
and turns $\phi_\Gamma(\gamma_1),\,\ldots,\,\phi_\Gamma(\gamma_{3g-3})$.
\end{enumerate}
\end{theorem}

We remark that ${\mathcal Q}\bigl(S_g,{\mathcal L},{\rm SU}(2,1)\bigr)$ is contained in the subset of
${\mathcal Q}\bigl(S_g,{\mathcal L},{\rm SL}(3,{\mathbb C})\bigr)$ where
\begin{enumerate}
\item[(1)] ${\rm tr}(A_j^{-1})=\overline{{\rm tr}(A_j)}$ for $j=1,\,\ldots,\,3g-3$;
\item[(2)] $\sigma_-(Y_k)=\overline{\sigma_+(Y_k)}$ for $k=1,\,\ldots,\,2g-2$;
\item[(3)] the bend and bulge parameters $\theta_\Gamma(\gamma_j)$, $s_\Gamma(\gamma_j)$ are all zero.
\end{enumerate}

We summarise the above results in the following table.
$$
\begin{array}{|l|l|l|}
\hline
G & \hbox{Parameters} & \hbox{Equations} \\
\hline
{\rm SO}_0(2,1) & {\rm tr}(A_j) & {\rm tr}(A_j)=\overline{{\rm tr}(A_j)}={\rm tr}(A_j^{-1})=\overline{{\rm tr}(A_j^{-1})} \\
& & \sigma_+(Y_k)=\sigma_-(Y_k) \hbox{ given by }\eqref{eq:shape-SL}  \\
& & Q(Y_k)\hbox{ repeated roots} \\
& t_\Gamma(\gamma_j) & \theta_\Gamma(\gamma_j)=s_\Gamma(\gamma_j)=\phi_\Gamma(\gamma_j)=0 \\
\hline
{\rm SO}(3;{\mathbb C}) & {\rm tr}(A_j) & {\rm tr}(A_j)={\rm tr}(A_j^{-1}),\ \overline{{\rm tr}(A_j)}=\overline{{\rm tr}(A_j^{-1})} \\
& & \sigma_+(Y_k)=\sigma_-(Y_k)\hbox{ given by }\eqref{eq:shape-SL}  \\
& & Q(Y_k) \hbox{ repeated roots} \\
& t_\Gamma(\gamma_j),\  \theta_\Gamma(\gamma_j) & s_\Gamma(\gamma_j)=\phi_\Gamma(\gamma_j)=0 \\
\hline
{\rm SL}(3,{\mathbb R}) & {\rm tr}(A_j) ,\ {\rm tr}(A_j^{-1}) & 
{\rm tr}(A_j)=\overline{{\rm tr}(A_j)}, \ {\rm tr}(A_j^{-1})=\overline{{\rm tr}(A_j^{-1})} \\
& \sigma_+(Y_k),\ \sigma_-(Y_k) & \sigma_+(Y_k)=\overline{\sigma_+(Y_k)},\ 
\sigma_-(Y_k)=\overline{\sigma_-(Y_k)} \\
& \hbox{root of }Q(Y_k) & \\
& t_\Gamma(\gamma_j),\ s_\Gamma(\gamma_j) & \theta_\Gamma(\gamma_j)=\phi_\Gamma(\gamma_j)=0 \\
\hline
{\rm SU}(2,1) & {\rm tr}(A_j) & {\rm tr}(A_j)=\overline{{\rm tr}(A_j^{-1})},\  {\rm tr}(A_j^{-1})=\overline{{\rm tr}(A_j)} \\
& \sigma_+(Y_k)  &  \sigma_+(Y_k)=\overline{\sigma_-(Y_k)},\ 
\sigma_-(Y_k)=\overline{\sigma_+(Y_k)} \\
& \hbox{root of }Q(Y_k) &  \\
& t_\Gamma(\gamma_j), \ \phi_\Gamma(\gamma_j) & \theta_\Gamma(\gamma_j)=s_\Gamma(\gamma_j)=0 \\
\hline
{\rm SL}(3,{\mathbb C}) & {\rm tr}(A_j),\ {\rm tr}(A_j^{-1}) & \\
& \sigma_+(Y_k),\ \sigma_-(Y_k) & \\
& \hbox{root of }Q(Y_k) &  \\
& t_\Gamma(\gamma_j),\  \theta_\Gamma(\gamma_j), \ s_\Gamma(\gamma_j), \ \phi_\Gamma(\gamma_j) & \\
\hline
\end{array}
$$

We note that the conditions above essentially characterise the representations of each pants group $\rho\bigl(\pi_1(Y_k)\bigr)$.
To see this, we use the following theorem of Acosta.

\begin{proposition}[Theorem~1.1 of Acosta \cite{Aco}]\label{prop:acosta}
Let $\Gamma$ be a finitely generated group and let $\rho:\Gamma\longrightarrow {\rm SL}(3,{\mathbb C})$ be an
irreducible representation of $\Gamma$. Then
\begin{enumerate}
\item[(1)] If ${\rm tr}(A)\in{\mathbb R}$ for all $A\in\rho(\Gamma)$ then $\rho(\Gamma)$ is conjugate to a representation
of $\gamma$ to ${\rm SL}(3,{\mathbb R})$.
\item[(2)] If ${\rm tr}(A^{-1})=\overline{{\rm tr}(A)}$ for all $A\in\rho(\Gamma)$ then $\rho(\Gamma)$ is conjugate to a
representation in ${\rm SU}(3)$ or ${\rm SU}(2,1)$. In particular, if $\rho(\Gamma)$ contains loxodromic maps then it
is conjugate to a representation in ${\rm SU}(2,1)$.
\end{enumerate}
\end{proposition}

Our result is

\begin{theorem}\label{thm-pants}
Let $\Gamma=\langle A,B,C : CBA=I\rangle$ be an irreducible subgroup of ${\rm SL}(3,{\mathbb C})$.
Let $\sigma_+$ and $\sigma_-$ be the shape invariants given by \eqref{eq:shape+} and
\eqref{eq:shape1}. 
\begin{enumerate}
\item[(1)] If ${\rm tr}(A)={\rm tr}(A^{-1})$, ${\rm tr}(B)={\rm tr}(B^{-1})$, ${\rm tr}(C)={\rm tr}(C^{-1})$, 
$\sigma_+=\sigma_-$ and $Q(\Gamma)$ has repeated roots, then up to conjugacy $\Gamma<{\rm SO}(3;{\mathbb C})$;
\item[(2)] If ${\rm tr}(A)$, ${\rm tr}(A^{-1})$, ${\rm tr}(B)$, ${\rm tr}(B^{-1})$, ${\rm tr}(C)$, ${\rm tr}(C^{-1})$, 
$\sigma_+$ and $\sigma_-$ are all real then up to conjugacy $\Gamma<{\rm SL}(3,{\mathbb R})$;
\item[(3)] If ${\rm tr}(A^{-1})=\overline{{\rm tr}(A)}$, ${\rm tr}(B^{-1})=\overline{{\rm tr}(B)}$, ${\rm tr}(C^{-1})=\overline{{\rm tr}(C)}$
and $\sigma_-=\overline{\sigma_+}$ then up to conjugacy $\Gamma<{\rm SU}(2,1)$.
\end{enumerate}
\end{theorem}

\section{Background on Fenchel-Nielsen coordinates}

\subsection{Geometrical Fenchel-Nielsen co-ordinates}\label{sec:geo-FN}

In \cite{F-N} Fenchel and Nielsen construct global coordinates for $\mathcal{T}(S_g)$, giving it the structure of a differentiable 
manifold. These coordinates depend on a choice of a curve system ${\mathcal L}$ on $S_g$ as given in
Definition~\ref{def-curve-system}.
Fenchel-Nielsen coordinates consist of $3g-3$ length coordinates and $3g-3$ twist coordinates. 
The length coordinates $\bigl(\ell_X(\gamma_1),\,\ldots,\,\ell_X(\gamma_{3g-3})\bigr)$ 
for a surface $X$ in Teichm\"uller space are the
hyperbolic lengths of the geodesics $\gamma_j$ measured using the hyperbolic structure on $X$.
In order to define the twists, we need to do a little more work. Consider $[\gamma_j]\in{\mathcal L}$. 
Then $\gamma_j$ either lies on the boundary of two distinct pairs of pants $Y$ and $Y'$ or it
corresponds to two boundary curves of a single pair of pants $Y$. Let $\alpha_j$ be a homotopically non-trivial simple closed 
curve in $Y\cup Y'$ (respectively $Y$) intersecting $\gamma_j$ minimally, that is in two points (respectively one point). 
We construct a piecewise geodesic curve in the homotopy class of $\alpha_j$ as follows. It consists of (a) two arcs $\delta_j$ and 
$\delta'_j$ (respectively a single arc $\delta_j$) contained in $\gamma_j$ and (b) two simple geodesic arcs $\beta_j\subset Y$,
$\beta'_j\subset Y'$ (respectively a single simple geodesic arc $\beta_j\subset Y$) meeting $\gamma_j$ orthogonally at 
their endpoints. 
Elementary hyperbolic geometry shows that $\delta_j$ and $\delta'_j$ have the same length. The twist $t_X(\gamma_j)$ is
the signed difference between the hyperbolic length of $\delta_j$ (measured using the hyperbolic structure on $X$) and the 
same length on some fixed reference surface $X_0$, and where the sign of $t_X(\gamma_j)$ is determined by a 
choice of orientation on $\gamma_j$.
For example we could take $X_0$ to be the surface where each $\delta_j$ has length zero,
that is where $\alpha_j$ and $\gamma_j$ are orthogonal simple closed geodesics, but this choice is not necessary. As a 
relative invariant, the twist is independent of the choice of $\alpha_j$. 

We define the Fenchel-Nielsen coordinates of $\mathcal{T}(S_g)$ with respect to a given 
curve system $\mathcal{L}=\{[\gamma_1],\ldots,[\gamma_{3g-3}]\}$ to be the map 
$FN:\mathcal{T}(S_g)\rightarrow\mathbb{R}_+^{3g-3}\times\mathbb{R}^{3g-3}$ given by
$$
FN_X=(\ell_X(\gamma_1),\ldots,\ell_X(\gamma_{3g-3}),t_X(\gamma_1),\ldots,t_X(\gamma_{3g-3})).
$$
The theorem of Fenchel and Nielsen says that these are global coordinates in the sense that two marked surfaces with 
distinct hyperbolic structures give different values of these parameters, and each value of the parameters gives a hyperbolic
structure on the surface.

\subsection{Algebraic Fenchel-Nielsen coordinates}\label{sec:alg-FN}

We now reinterpret Fenchel-Nielsen coordinates in terms of 
${\rm Hom}(\pi_1(S_g), {\rm SL}(2,{\mathbb R}))// {\rm SL}(2,{\mathbb R})$, 
the deformation space of representations of $\pi_1(S_g)$ to ${\rm SL}(2,{\mathbb R})$ up to conjugacy. 
This is equivalent to the Fuchsian space from Definition~\ref{def-Fuchs-space} under the identification of
${\rm PSL}(2,{\mathbb R})$ and ${\rm SO}_0(2,1)$, see Section~\ref{sec-irred}.
Let $Y$ be one of the
pairs of pants, that is $Y$ is a component of $S_g\backslash\cup_{j=1}^{3g-3}\gamma_j$. 
Then $\pi_1(Y)$ is a free group on two generators. 
We can take the generators to be the homotopy classes of curves corresponding to two of the boundary curves. Then the third
boundary curve corresponds to the product of these two generators. In fact, it is more convenient to regard $\pi_1(Y)$
as having three generators, corresponding to the three boundary components, with a single relation that their product is
the identity. That is, if the homotopy classes of $\partial Y$ are $[\alpha]$, $[\beta]$, $[\gamma]$ then 
$$
\pi_1(Y)=\bigl\langle [\alpha],\,[\beta],\,[\gamma]\, :\,  [\gamma][\beta][\alpha]=id\bigr\rangle.
$$
Consider a representation $\rho:\pi_1(Y)\longrightarrow {\rm SL}(2,{\mathbb R})$ and write $A=\rho([\alpha])$,
$B=\rho([\beta])$ and $C=\rho([\gamma])$. We then have $CBA=I$. In other words,
$$
\rho\bigl(\pi_1(Y)\bigr)=\Gamma=\langle A,\,B,\,C\,:\,CBA=I\rangle.
$$
A classical theorem of Fricke and Vogt (see Theorem~\ref{thm-fv} below for a precise statement) says that if 
$\rho$ is irreducible then $\rho\bigl(\pi_1(Y)\bigr)$ is
completely determined up to conjugacy by ${\rm tr}(A)$, ${\rm tr}(B)$ and ${\rm tr}(C)$. Furthermore, in order for $A$, $B$
and $C$ to represent the boundaries of a pair of pants, they must all be hyperbolic elements, their axes should be disjoint
and not separate each other. In this case, a well known result, see Gilman and Maskit \cite{G-M}, says that
${\rm tr}(A){\rm tr}(B){\rm tr}(C)<0$. We therefore normalise the representation be supposing that each of 
${\rm tr}(A)$, ${\rm tr}(B)$, ${\rm tr}(C)$ lies in the interval $(-\infty, -2)$. We then note that
${\rm tr}(A)=-2\cosh\bigl(\ell_X(\alpha)/2\bigr)$ where, as above, $\ell_X(\alpha)$ is the length with respect 
to the hyperbolic metric on $X$ of the geodesic $\alpha$ in the homotopy class $[\alpha]$, and similarly for 
${\rm tr}(B)=-2\cosh\bigl(\ell_X(\beta)/2\bigr)$ and ${\rm tr}(C)=-2\cosh\bigl(\ell_X(\gamma)/2\bigr)$.

We now discuss the algebraic interpretation of how to attach two pairs of pants and how to close a handle, see 
Parker and Platis \cite{PP}. First consider attaching two pairs of pants. 
Suppose that $Y$ and $Y'$ are two pairs of pants with a hyperbolic structure and geodesic boundary. 
Write $\Gamma=\rho\bigl(\pi_1(Y)\bigr)$ and $\Gamma'=\rho\bigl(\pi_1(Y')\bigr)$, with 
$$
\Gamma=\langle A,\,B,\,C\,:\,CBA=I\rangle, \quad 
\Gamma'=\langle A',\,B',\,C'\,:\,C'B'A'=I\rangle,
$$ 
for the images of their fundamental groups under $\rho$. We want to glue them along 
the boundary curves $\alpha$ and $\alpha'$. In order to do so, $\alpha$ and $\alpha'$ must have the same length and 
opposite orientation. Algebraically, this says that if $A=\rho([\alpha])$ and $A'=\rho([\alpha'])$ then $A'$ is conjugate to
$A^{-1}$. Without loss of generality, we assume $A'=A^{-1}$. This gives a representation of $\pi_1(Y\cup_\alpha Y')$ as 
the free product with amalgamation along $\langle A\rangle=\langle A'\rangle$:
\begin{eqnarray*}
\rho\bigl(\pi_1(Y\cup_\alpha Y')\bigr)
& = & \Gamma *_{\langle A\rangle} \Gamma' \\
& = & \langle A,B,C\,:CBA=I\rangle *_{\langle A\rangle}\langle A',B',C\,:\,C'B'A'=I\rangle \\
& = & \langle B,C,B',C'\,:\, CBC'B'=I\rangle.
\end{eqnarray*}
To obtain the relation, we combine $AA'=CBA=C'B'A'=I$:
$$
(CB)(C'B')=A^{-1}{A'}^{-1}=(A'A)^{-1}=I.
$$

Now consider closing a handle. Suppose $Y$ is a pair of pants
with a hyperbolic structure and geodesic boundary. Write $\Gamma=\langle A,\,B,\,C\,:\,CBA=I\rangle$ for
the image of its fundamental group under $\rho$. We want to glue two boundary components $\alpha$ and $\beta$. 
We write them as $A=\rho([\alpha])$ and $B=\rho([\beta])$.  
As above, this means $\alpha$ and $\beta$ have the same length and opposite orientation. Algebraically, this
means $B$ is conjugate to $A^{-1}$. Suppose that the conjugating map is denoted by $D$, so
$B=DA^{-1}D^{-1}$. We can therefore form the HNN extension
\begin{eqnarray*}
\Gamma *_{\langle D\rangle}
& = & \langle A,(DA^{-1}D^{-1}),C\,:\, C(DA^{-1}D^{-1})A=I\rangle *_{\langle D\rangle} \\
& = & \langle A,C,D\,:\, C[D,A^{-1}]=I\rangle.
\end{eqnarray*}

Suppose ${\mathcal L}$ is chosen in such a way $S_g$ is obtained using the following process. First, attach 
$2g-2$ pairs of pants to form a $2g$-holed sphere, so that the $2g$ boundary curves form $g$ pairs where each pair is in the
same pair of pants. Secondly, close the $g$ handles by identifying curves in the same pair of pants. Using induction
on the attaching step above, we see that the $2g$-holed sphere is represented by a group 
$$
\langle B_1,C_1,\,\ldots B_g,C_g : C_1B_1\cdots B_gC_g=I\rangle.
$$
Closing the $g$ handles replaces each pair $B_k,C_k$ with a commutator. Thus we obtain the standard presentation 
for a surface group. In what follows, it is not necessary to make this choice. The main difference would be that we would close
handles by identifying boundary curves in different pairs of pants. 

We now discus how to interpret the Fenchel-Nielsen twist $t_\Gamma(\alpha)$ parameter around $\alpha$. In the above construction
we made a choice when we performed the gluing. The ambiguity in that choice is exactly given by an element $K$ of the
centraliser of $A$. That is, $K$ commutes with $A$ and so must have the same eigenvectors. In the first case,
we can conjugate $\langle A',B',C'\,:\, C'B'A'=I\rangle$ by $K$ to obtain
$$
\Gamma*_{\langle A\rangle}K\Gamma'K^{-1}=\langle B,C,KB'K^{-1},KC'K^{-1}\,:\, CB(KC'K^{-1})(KB'K^{-1})=I\rangle.
$$
In the second case, we replace the conjugating map $D$ with $DK$ to obtain
$$
\Gamma *_{\langle DK\rangle}=\langle A,C,DK\,:\, C[DK,A^{-1}]C=I\rangle.
$$
Since $K$ commutes with $A$ we still have $A'=KA'K^{-1}=KA^{-1}K^{-1}=A^{-1}$ and $B=DKA^{-1}(DK)^{-1}=DA^{-1}D^{-1}$.

In order to relate $K$ to the twist $t_\Gamma(\alpha)$ we could use the trace of $K$ in the same way that we 
related $\ell_X(\alpha)$ to ${\rm tr}(A)$. However, that does not capture the sign of the twist. 
Instead, we use an eigenvalue.  Since $A$ is 
hyperbolic (loxodromic) and $K$ is in its centraliser $Z(A)$, they must have the same eigenvectors. The eigenvalues of $A$ are
$-e^{\ell_X(\alpha)/2}$ and $-e^{-\ell_X(\alpha)/2}$ and we denote the associated eigenvectors by 
${\bf v}_+(A)$ and ${\bf v}_-(A)$, respectively. We suppose
${\rm tr}(K)>0$ and define the twist by saying the eigenvlaue $\lambda_K$ of $K$ associated to the eigenvector ${\bf v}_+(A)$ 
is $e^{t_\Gamma(\alpha)/2}$. Note that the choice of this eigenvalue is equivalent to a choice of orientation of $\alpha$. 
The twist parameter could also be parametrised using traces. For example, in \cite{Mas} Maskit computes the 
Fenchel-Nielsen coordinates explicitly using matrices.

In the above definition, we made a choice of $Y$ rather than $Y'$. We now show $t_\Gamma(\alpha)$ is independent of this
choice. Swapping the roles of $Y$ and $Y'$, means we conjugate $\langle A,B,C\,:\,CBA=I\rangle$ by $K^{-1}$ to obtain
$$
\langle K^{-1}BK,K^{-1}CK,B',C'\,:\, (K^{-1}CK)(K^{-1}BK)C'B'=I\rangle.
$$
Thus the twist is given by the eigenvalue $\lambda_{K^{-1}}$ of $K^{-1}$ associated to ${\bf v}_+(A')={\bf v}_-(A)$, and this 
eigenvalue is again $e^{t_\Gamma(\alpha)}$. Thus this definition of $t_\Gamma(\alpha)$ does not depend on a 
choice of $Y$ or $Y'$.
Similarly, when closing a handle the definition does not depend on the choices we made. 

Combining all of the pairs of pants associated to the curve system $\mathcal{L}=\{[\gamma_1],\ldots,[\gamma_{3g-3}]\}$ on $S_g$,
we obtain algebraic Fenchel Nielsen coordinates associated to the representation $\Gamma=\rho\bigl(\pi_1(S_g)\bigr)$ as 
$$
FN_\rho=\bigl({\rm tr}(A_1),\,\ldots,\,{\rm tr}(A_{3g-3}),t_\Gamma(\gamma_1),\,\ldots,\,t_\Gamma(\gamma_{3g-3})\bigr)
$$
where $A_j=\rho([\gamma_j])$ and $K_j\in Z(A_j)$ has eigenvalue $e^{t_\Gamma(\gamma_j)/2}$ associated to the 
eigenvector of $A_j$
with eigenvalue of largest absolute value. 

Next, we mention some examples of our interest for the develop of this project:

\begin{itemize}
	\item For $G={\rm SL}(2,\mathbb{C})$ Kourouniotis \cite{Kou} and Tan \cite{Tan} (independently) generalise the 
	Fenchel-Nielsen coordinates for quasi-Fuchsian representations.
	
	\item For $G={\rm SL}(3,\mathbb{R})$ Goldman in \cite{Gol} generalises the Fenchel-Nielsen coordinates for the 
	space of convex projective structures.
	
	\item For $G={\rm SU}(2,1)$ Parker and Platis in \cite{PP} generalise the Fenchel-Nielsen coordinates for the 
	space of complex hyperbolic quasi-Fuchsian representations.
\end{itemize}

We remark that the space of convex projective structures studied by Goldman is the Hitchin component of
the ${\rm SL}(3,{\mathbb R})$ character variety of $\pi_1(S_g)$ \cite{Choi-Gol}. In his PhD thesis \cite{Zhang}, Tengren Zhang 
defined Fenchel-Nielsen coordinates for the Hitchin component of the ${\rm SL}(n,{\mathbb R})$ charaxter variety for all
$n\ge 2$.

In this paper we generalise Fenchel-Nielsen coordinates for the case when $G={\rm SL}(3,\mathbb{C})$. All of the four 
cases mentioned 
above give representations of $\pi_1(\Sigma_g)$ to subgroups of ${\rm SL}(3,\mathbb{C})$ our coordinates should be a direct 
generalisation in each case. Since we lose many of the geometric features, we are going to use the algebraic version using as a 
main tools traces and eigenvalues of the representations. 

Discrete subgroups of ${\rm SL}(3,{\mathbb C})$ and their
action on ${\mathbb C}{\mathbb P}^2$ have been studied by Cano, Navarrete and Seade \cite{C-N-S}. 
In particular, Cano, Parker and Seade \cite{C-P-S} showed that  the action on ${\mathbb C}{\mathbb P}^2$ of a discrete 
subgroup of ${\rm SO}_0(2,1)$ (the image under the irreducible representation of ${\rm PSL}(2,{\mathbb R})$) has
three components in its domain of discontinuity and a connected limit set. It would be interesting to investigate how
this varies for a surface group as we vary the representation away from ${\rm SO}_0(2,1)$. Our Fenchel-Nielsen
coordinates give a framework for doing so.

\section{Complex projective Fenchel-Nielsen coordinates}

In this section we are going to mimic the construction from Section~\ref{sec:alg-FN} but for the space
${\rm Hom}(\pi_1(S_g), {\rm SL}(3,\mathbb{C}))// {\rm SL}(3,\mathbb{C})$ of representations of $\pi_1(S_g)$ to
${\rm SL}(3,C)$ up to conjugation. The classical trichotomy of elements of ${\rm SL}(2,\mathbb{C})$, 
can be generalised to ${\rm SL}(3,{\mathbb C})$ as follows, see Theorem~4.3.1 on page 112 of 
Cano, Navarrete and Seade \cite{C-N-S}:

\begin{theorem} \label{thm-SL3C-class}
Every element in ${\rm SL}(3,\mathbb{C})\backslash\{I\}$ is one and only one of the following classes: 
elliptic (diagonalizable whit unitary eigenvalues), parabolic (non-diagonalizable) or 
loxodromic (diagonalizable with non-unitary eigenvalues).
	\begin{enumerate}		
		\item [(i)] An elliptic transformation belongs to one and only one of the following classes: regular (it has pairwise 
		different eigenvalues) or conjugate to a complex reflection (two eigenvalues are repeated).		
		\item [(ii)] A parabolic transformation  belongs to one and only one of the following classes: unipotent 
		(it has eigenvalues equal to one), or ellipto-parabolic (it is not unipotent).
		
		\item [(iii)] A loxodromic element belongs to one and only one of the following four classes: 
		loxo-parabolic (only have two eigenvalues with different modulus), complex homothety, 
		screw (different eigenvalues but two of them have the same modulus) or strongly loxodromic 
		(different eigenvalues with different modulus).
	\end{enumerate}
\end{theorem}

We are going to be interested in irreducible, faithful and discrete representations of the fundamental group of a surface in 
${\rm SL}(3,\mathbb{C})$ where all the elements of the representation will be strongly loxodromic maps.
 
 Using a result of Navarrete, Theorem 7.3 of \cite{Nav}, see also Theorem 4.3.3 of Cano-Navarrete-Seade \cite{C-N-S}
we can use ${\rm tr}(A)$ and ${\rm tr}(A^{-1})$ to determine whether or not
$A\in{\rm SL}(3,{\mathbb C})$ is strongly loxodromic. 

\begin{proposition}\label{prop:nav-tr}
Define
$$
F(x,y)=x^2y^2-4(x^3+y^3)+18xy-27.
$$
The map $A\in{\rm SL}(3,{\mathbb C})$ is strongly loxodromic if and only if
\begin{enumerate}
\item[(1)] either ${\rm tr}(A^{-1})=\overline{{\rm tr}(A)}$ and $F\bigl({\rm tr}(A),{\rm tr}(A^{-1})\bigr)>0$, 
\item[(2)] or ${\rm tr}(A^{-1})\neq \overline{{\rm tr}(A)}$ and $F\bigl{(\rm tr}(A),{\rm tr}(A^{-1})\bigr)\neq 0$.
\end{enumerate}
\end{proposition}

\subsection{Polynomial matrix relations}

The goal of this section is to extend the work of Fricke and Vogt, see Theorem A of Goldman \cite{Gol1}, to two generator 
subgroups of ${\rm SL}(3,\mathbb{C})$ following the work of Lawton \cite{Law}, see also Will \cite{Will2} and Parker \cite{Par}. 

We define two polynomials in eight variables:
\begin{eqnarray}
\label{eq:S-x} S_0({\bf x}) & = & x_1x_5+x_2x_6+x_3x_7+x_4x_8 +x_1x_2x_5x_6\\
\nonumber		&& \quad -x_1x_2x_7-x_1x_4x_6-x_2x_5x_8-x_3x_5x_6-3, \\
\label{eq:P-x} P_0({\bf x}) 
& = & x_1^2x_2x_5^2x_6+x_1x_2^2x_5x_6^2 
+x_1^2x_2^2x_3+x_5^2x_6^2x_7+x_1^2x_6^2x_8+x_2^2x_4x_5^2 \\
\nonumber && \quad -x_1^2x_2x_5x_7-x_1x_3x_5^2x_6-x_1^2x_4x_5x_6-x_1x_2x_5^2x_8 \\
\nonumber && \quad -x_2^2x_5x_6x_8-x_1x_2x_4x_6^2-x_1x_2^2x_6x_7-x_2x_3x_5x_6^2 \\
\nonumber && \quad -x_1^3x_2x_6-x_2x_5^3x_6-x_1x_2^3x_5-x_1x_6^3x_5 \\
\nonumber && \quad -x_1x_2x_3x_4x_5-x_1x_5x_6x_7x_8-x_1x_2x_3x_6x_8-x_2x_4x_5x_6x_7 \\
\nonumber && \quad +x_1^2x_2x_8+x_4x_5^2x_6+x_1^2x_3x_6+x_2x_5^2x_7+x_1^2x_4x_7+x_3x_5^2x_8 \\
\nonumber && \quad +x_1x_2^2x_4+x_5x_6^2x_8+x_2^2x_3x_5+x_1x_6^2x_7+x_2^2x_7x_8+x_3x_4x_6^2 \\
\nonumber && \quad +x_3^2x_4x_5+x_1x_7^2x_8+x_3^2x_6x_8+x_2x_4x_7^2 \\
\nonumber && \quad +x_1x_3x_4^2+x_5x_7x_8^2+x_4^2x_6x_7+x_2x_3x_8^2 \\
\nonumber & & \quad -2x_1x_2x_3^2-2x_5x_6x_7^2-2x_2x_4^2x_5-2x_1x_6x_8^2\\
\nonumber && \quad +x_1x_2x_5x_6+x_1x_3x_5x_7+x_1x_4x_5x_8 \\
\nonumber && \quad +x_2x_3x_6x_7+x_2x_4x_6x_8+x_3x_4x_7x_8 \\ 
\nonumber && \quad +x_1^3+x_2^3+x_3^3+x_4^3+x_5^3+x_6^3+x_7^3+x_8^3 \\
\nonumber && \quad -3x_1x_3x_8-3x_4x_5x_7-3x_2x_3x_4-3x_6x_7x_8 \\
\nonumber && \quad +3x_1x_4x_6+3x_2x_5x_8+3x_1x_2x_7+3x_3x_5x_6 \\
\nonumber && \quad-6x_1x_5-6x_2x_6-6x_3x_7-6x_4x_8+9.
\end{eqnarray}

\begin{theorem}[Lawton \cite{Law}]\label{thm-lawton}
Let ${\bf x}=(x_1,\,\ldots,\, x_8)$ be any vector in ${\mathbb C}^8$. Then:
\begin{enumerate}
\item[(1)] 
There exist $A,\,B\in{\rm SL}(3,{\mathbb C})$ so that
\begin{equation}\label{eq:8-traces}
	\begin{array}{llll} x_1={\rm tr}(A),& x_2={\rm tr}(B),& x_3={\rm tr}(AB),& x_4={\rm tr}(A^{-1}B), \\
	x_5={\rm tr}(A^{-1}), & x_6={\rm tr}(B^{-1}), & x_7={\rm tr}(B^{-1}A^{-1}), & x_8={\rm tr}(B^{-1}A).
    \end{array}
\end{equation}
\item[(2)] If $A$ and $B$ are as in part (1) then 
$$
{\rm tr}[A,B]+{\rm tr}[B,A]=S_0({\bf x}),\quad {\rm tr}[A,B]{\rm tr}[B,A]=P_0({\bf x})
$$
where $S_0$ and $P_0$ are the polynomials defined by \eqref{eq:S-x} and \eqref{eq:P-x} evaluated
at the point ${\bf x}$ given by \eqref{eq:8-traces}. 

In particular, ${\rm tr}[A,B]$ and ${\rm tr}[B,A]$ are the roots of the polynomial
	\begin{equation}\label{eq:Q-x}
		Q_0(X)=X^2-S_0({\bf x})X+P_0({\bf x})
	\end{equation}
whose coefficients only depend on the eight traces in \eqref{eq:8-traces}
\item[(3)] Let $A,B\in {\rm SL}(3,\mathbb{C})$ be as in part (1). If the group $\langle A,B\rangle$ is irreducible, 
then it is determined up to conjugation within ${\rm SL}(3,\mathbb{C})$ by
	$$
	\begin{array}{lllll} {\rm tr}(A),& {\rm tr}(B),& {\rm tr}(AB),& {\rm tr}(A^{-1}B), & {\rm tr}[A,B], \\
	{\rm tr}(A^{-1}), & {\rm tr}(B^{-1}), & {\rm tr}(B^{-1}A^{-1}), & {\rm tr}(B^{-1}A).
    \end{array}
    $$
In other words, if $\langle A,B\rangle$ is irreducible then it is determined by the point ${\bf x}\in{\mathbb C}^8$
from \eqref{eq:8-traces} together with a choice of root of the quadratic polynomial \eqref{eq:Q-x}.
\end{enumerate}
\end{theorem}

Note that part (3) means that if $\langle A', B'\rangle$ is any representation (possibly reducible) so that the eight traces 
in \eqref{eq:8-traces} agree with those of $\langle A,B\rangle$ and that we choose the same root
of the quadratic \eqref{eq:Q-x} for both groups, then $\langle A',B'\rangle$ is irreducible and conjugate to $\langle A,B\rangle$.

If $\langle A,B\rangle$ is reducible, then $A$ and $B$ share an eigenvector. It is clear that this vector is an eigenvector
of the commutator $[A,B]$ with eigenvalue 1; see Lemma~\ref{lem-1-evalue} below. 
From this it follows that ${\rm tr}[A,B]={\rm tr}[B,A]$ and so
$\langle A,B\rangle$ is in the branching locus of the quadratic $Q_0$. That is $S_0^2-4P_0=0$. We will see later that
the converse is not true, namely there are irreducible groups in the branching locus, for example when $\langle A,B\rangle$
is in the irreducible representation of ${\rm SL}(2,{\mathbb C})$ to ${\rm SL}(3,{\mathbb C})$.

\medskip

Given $\langle A,B,C:CBA=I\rangle$, the coordinates of Fricke and Vogt for ${\rm SL}(2,{\mathbb R})$ representations 
of this group are symmetric in cyclic permutation of $A$, $B$ and $C$. This is not the case with Lawton's parameters for 
${\rm SL}(3,{\mathbb C})$ representations of the group. We now show how to symmetrise Lawton's parameters.

First we observe that 
$$
x_3={\rm tr}(AB)={\rm tr}(C^{-1}),\quad x_7={\rm tr}(B^{-1}A^{-1})={\rm tr}(C).
$$
Symmetrising $x_4={\rm tr}(A^{-1}B)$ and $x_8={\rm tr}(B^{-1}A)$ is slightly more difficult.

\begin{lemma}\label{lem:char-poly}
	Let $A\in {\rm SL}(3,\mathbb{C})$. Then the characteristic polynomial of $A$ is 
	\begin{equation}\label{eqn:1.5}
		\chi_A(x)= x^3-{\rm tr}(A)x^2+{\rm tr}(A^{-1})x-1
	\end{equation}
\end{lemma}
\begin{proof}
	Let $\lambda_{1}$, $\lambda_{2}$, $\lambda_{3}$ be the eigenvalues of $A$. 
	Then $\lambda_{1}\lambda_{2}\lambda_{3}={\rm det}(A)=1$ which is the constant term of the characteristic polynomial. 
	We know that the quadratic term is $\lambda_{1}+\lambda_{2}+\lambda_{3}={\rm tr}(A)$. 
	Using $\lambda_{1}\lambda_{2}\lambda_{3}=1$  the linear term is
	$$
	\lambda_{2}\lambda_3+\lambda_{1}\lambda_{3}+\lambda_{1}\lambda_{2}
	=\lambda_1^{-1}+\lambda_{2}^{-1}+\lambda_{3}^{-1}.
	$$
	Since $\lambda_{1}^{-1}$, $\lambda_{2}^{-1}$, $\lambda_{3}^{-1}$ are the eigenvalues of $A^{-1}$, we see that
	the linear term is ${\rm tr}(A^{-1})$, as claimed.
	\end{proof}

The lemma below was proved in \cite{Par} for ${\rm SU}(2,1)$, but in fact is valid for ${\rm SL}(3,{\mathbb C})$.

\begin{lemma}\label{Lem2.5}
	Let $A,\,B,\,C\in {\rm SL}(3,\mathbb{C})$ with $CBA=I$, then
	\begin{eqnarray}
		{\rm tr}(A^{-1}B)-{\rm tr}(A^{-1}){\rm tr}(B) 
		& = & {\rm tr}(C^{-1}A)-{\rm tr}(C^{-1}){\rm tr}(A) \label{eq:shape1}\\
		& = & {\rm tr}(B^{-1}C)-{\rm tr}(B^{-1}){\rm tr}(C), \nonumber \\
				{\rm tr}(AB^{-1})-{\rm tr}(B^{-1}){\rm tr}(A) 
		& = & {\rm tr}(CA^{-1})-{\rm tr}(A^{-1}){\rm tr}(C) \label{eq:shape2}\\
		& = & {\rm tr}(BC^{-1})-{\rm tr}(C^{-1}){\rm tr}(B). \nonumber
	\end{eqnarray}
\end{lemma}
\begin{proof}
From Lemma~\ref{lem:char-poly} and the Cayley-Hamilton theorem we have 
\begin{equation}\label{eqn:1.6}
	A^3-{\rm tr}(A)A^2+{\rm tr}(A^{-1})A-I=O
\end{equation}
	We multiply equation \eqref{eqn:1.6} on the left by $BA^{-1}$ to get
	\begin{equation*}
		BA^2-{\rm tr}(A)BA+{\rm tr}(A^{-1})B-BA^{-1}=O.
	\end{equation*}
Since $CBA=I$ then $C^{-1}=BA$ and we substitute
\begin{equation*}
	C^{-1}A-{\rm tr}(A)C^{-1}+{\rm tr}(A^{-1})B-BA^{-1}=O
\end{equation*}
and taking traces we get
\begin{equation*}
	{\rm tr}(C^{-1}A)-{\rm tr}(A){\rm tr}(C^{-1})+{\rm tr}(A^{-1}){\rm tr}(B)-{\rm tr}(BA^{-1})=0.
\end{equation*}
Rearranging gives
$$
{\rm tr}(C^{-1}A)-{\rm tr}(C^{-1}){\rm tr}(A)={\rm tr}(A^{-1}B)-{\rm tr}(A^{-1}){\rm tr}(B).
$$
Cyclically permuting $A$, $B$ and $C$ gives
$$
{\rm tr}(A^{-1}B)-{\rm tr}(A^{-1}){\rm tr}(B)={\rm tr}(B^{-1}C)-{\rm tr}(B^{-1}){\rm tr}(C). 
$$
This gives \eqref{eq:shape1}

Starting from \eqref{eqn:1.6} and multiplying on the right by $A^{-1}C$, a similar argument gives
\eqref{eq:shape2}.
\end{proof}

\medskip

\begin{definition}\label{def-shape-invariants}
Define the {\em shape invariants} of the triple $A,\,B,\,C$, or of the group 
$\Gamma=\langle A,\,B,\,C\,:\, CBA=I\rangle$ they generate, as
\begin{eqnarray}
\sigma_+=\sigma_+(A,B,C)=\sigma_+(\Gamma) &:= & {\rm tr}(A^{-1}B)-{\rm tr}(A^{-1}){\rm tr}(B), \label{eq:shape+} \\
\sigma_-=\sigma_-(A,B,C)=\sigma_-(\Gamma) &:= &{\rm tr}(B^{-1}A)-{\rm tr}(B^{-1}){\rm tr}(A). \label{eq:shape-}
\end{eqnarray}	
\end{definition}
	
From Lemma~\ref{Lem2.5}, we see that
the shape invariants are invariant under cyclic permutation
of $A$, $B$ and $C$.

We also remark that 
$$
[A,B]=ABA^{-1}B^{-1}=ABC
$$ 
and 
$$
[B,A]=BAB^{-1}A^{-1}=C^{-1}B^{-1}A^{-1}=(ABC)^{-1}.
$$ 
It is easy to see that this implies 
$$
{\rm tr}[A,B]={\rm tr}[B,C]={\rm tr}[C,A],\quad {\rm tr}[B,A]={\rm tr}[C,B]={\rm tr}[A,C].
$$
This implies that equation \eqref{eq:Q-x} is invariant under cyclic permutation of $A$, $B$, $C$ and so this must be true
of the polynomials $S_0({\bf x})$ and $P_0({\bf x})$. Following Proposition~4.10 of \cite{Par}, the easiest way to see this is to use 
the variables
${\bf y}=(y_1,\,\ldots,\, y_8)$ where
\begin{equation}
	\begin{array}{llll} 
	y_1={\rm tr}(A),& y_2={\rm tr}(B),& y_3={\rm tr}(C),& y_4=\sigma_+(A,B,C), \\ 
	y_5={\rm tr}(A^{-1}), & y_6={\rm tr}(B^{-1}), & y_7={\rm tr}(C^{-1}), & y_8=\sigma_-(A,B,C).
    \end{array}
\end{equation}
In particular, $x_3=y_7$ and 
$$
x_4={\rm tr}(A^{-1}B)=\sigma_+(A,B,C)+{\rm tr}(A^{-1}){\rm tr}(B)=y_4+y_2y_5.
$$
Using this substitution, we define
\begin{eqnarray*}
S(y_1,\,\ldots,\,y_8) & = & S_0\bigl(y_1,y_2,y_7,(y_4+y_2y_5),y_6,y_7,y_3,(y_8+y_1y_6)\bigr), \\
P(y_1,\,\ldots,\,y_8) & = & P_0\bigl(y_1,y_2,y_7,(y_4+y_2y_5),y_6,y_7,y_3,(y_8+y_1y_6)\bigr).
\end{eqnarray*}
Specifically, peforming the substitution we obtain:
\begin{eqnarray}
\label{eq:S-y} S({\bf y}) & = & y_1y_5+y_2y_6+y_3y_7+y_4y_8 -y_1y_2y_3-y_5y_6y_7-3, \\
\label{eq:P-y}P({\bf y}) & = & y_1y_2y_3y_5y_6y_7 \\
\nonumber && \ 
+ y_1^2y_2^2y_7 + y_3y_5^2y_6^2
+y_1^2y_3^2y_6  + y_2y_5^2y_7^2
+y_2^2y_3^2y_5  + y_1y_6^2y_7^2 \\
\nonumber && \ 
+y_1y_2y_5y_6+y_2y_3y_6y_7+y_1y_3y_5y_7 \\
\nonumber && \ 
 -2y_1y_2y_7^2 - 2y_3^2y_5y_6
 -2y_1y_3y_6^2 - 2y_2^2y_5y_7
 -2 y_2y_3y_5^2 - 2 y_1^2y_6y_7 \\
 \nonumber && \ 
 +y_1^3+y_2^3+y_3^3+y_5^3+y_6^3+y_7^3 \\
 \nonumber && \ 
 + 3y_1y_2y_3+3y_5y_6y_7-6y_1y_5 -6y_2y_6 -6y_3y_7 \\
 \nonumber && \ 
 +y_1y_2y_4y_5y_7 + y_1y_3y_4y_6y_7 + y_2y_3y_4y_5y_6 \\
 \nonumber && \ 
 +y_1y_2^2y_4 + y_4y_5^2y_6 + y_1^2y_3y_4 + y_4y_5y_7^2+y_2y_3^2y_4+y_4y_6^2y_7 \\
 \nonumber &&  \ 
 +y_1y_3y_5y_6y_8 + y_2y_3y_5y_7y_8 + y_1y_2y_6y_7y_8 \\
 \nonumber && \ 
 +y_5y_6^2y_8 + y_1^2y_2y_8 + y_3y_5^2y_6+y_1y_3^2y_8 + y_6y_7^2y_8+y_2^2y_3y_8 \\
 \nonumber && \ 
 +(y_4^2-3y_8)(y_1y_7+y_2y_5+y_3y_6) \\
 \nonumber && \ 
 +(y_8^2-3y_4)(y_1y_6+y_2y_7+y_3y_5) \\
 \nonumber && \ 
 +y_4y_8(y_1y_5+y_2y_6+y_3y_7-6)+y_4^3+y_8^3+9.
\end{eqnarray}
It is easy to see that cyclic permutation of $A$, $B$ and $C$ gives a permutation of $(y_1,\,\ldots,\,y_8)$
that preserves $S({\bf y})$ and $P({\bf y})$. 
Therefore, we can rewrite Lawton's theorem as follows, which generalises the theorem of Fricke and Vogt to our case:

\begin{theorem}\label{thm-lawton-new}
Let ${\bf y}=(y_1,\,\ldots,\, y_8)$ be any vector in ${\mathbb C}^8$. Then:
\begin{enumerate}
\item[(1)] 
There exist $A,\,B,\,C\in{\rm SL}(3,{\mathbb C})$ with $CBA=I$ so that
\begin{equation}\label{eq:8-traces-y}
	\begin{array}{llll} 
	y_1={\rm tr}(A),& y_2={\rm tr}(B),& y_3={\rm tr}(C),& y_4=\sigma_+(A,B,C), \\
	y_5={\rm tr}(A^{-1}), & y_6={\rm tr}(B^{-1}), & y_7={\rm tr}(C^{-1}), & y_8=\sigma_-(A,B,C),
    \end{array}
\end{equation}
where $\sigma_+$ and $\sigma_-$ are given by \eqref{eq:shape+} and \eqref{eq:shape-}.
\item[(2)] If $A,\,B,\,C$ are as in part (1) then 
$$
{\rm tr}[A,B]+{\rm tr}[B,A]=S({\bf y}),\quad {\rm tr}[A,B]{\rm tr}[B,A]=P({\bf y})
$$
where $S$ and $P$ are the polynomials defined by \eqref{eq:S-y} and \eqref{eq:P-y} evaluated
at the point ${\bf y}$ given by \eqref{eq:8-traces-y}. 

In particular, ${\rm tr}[A,B]$ and ${\rm tr}[B,A]$ are the roots of the polynomial
	\begin{equation}\label{eqn:Q-y}
		Q(X)=X^2-S({\bf y})X+P({\bf y})
	\end{equation}
whose coefficients only depend on the traces and shape invariants in \eqref{eq:8-traces-y}.
\item[(3)] Let $A,B,C \in {\rm SL}(3,\mathbb{C})$ with $CBA=I$ be as in part (1).
If the group generated by $A$, $B$ and $C$ is irreducible, 
then it is determined up to conjugation within ${\rm SL}(3,\mathbb{C})$ by
	$$
	\begin{array}{lllll} {\rm tr}(A),& {\rm tr}(B),& {\rm tr}(C),& \sigma_+(\Gamma), & {\rm tr}[A,B], \\
	{\rm tr}(A^{-1}), & {\rm tr}(B^{-1}), & {\rm tr}(C^{-1}), & \sigma_-(\Gamma).
    \end{array}
    $$
In other words, if $\langle A,B\rangle$ is irreducible then it is determined by the eight traces
from \eqref{eq:8-traces-y} together with a choice of root of the quadratic polynomial 
$Q$, from \eqref{eqn:Q-y}.
\end{enumerate}
\end{theorem}

\subsection{Twist-bend-buldge-turn parameter.}

Once again, we use the free product with amalgamation of $\Gamma=\rho\bigl(\pi_1(Y)\bigr)$ and 
$\Gamma'=\rho\bigl(\pi_1(Y')\bigr)$ along $A'=A^{-1}$ and we use the HNN extension to glue 
$\Gamma=\rho\bigl(\pi_1(Y)\bigr)$ along two conjugate peripheral curves $A$ and $B=DA^{-1}D^{-1}$ to obtain
\begin{eqnarray*}
\Gamma *_{\langle A\rangle} \Gamma' & = & \langle B,C,B',C'\,:\, CBC'B'=I\rangle, \\
\Gamma *_{\langle D\rangle} & = & \langle A,C,D\,:\,C[D,A^{-1}]=I\rangle.
\end{eqnarray*}

As in Section~\ref{sec:alg-FN}, there are further parameters that capture the freedom we have when 
taking the free product with amalgamation and the HNN extension. Namely, in each case we take $K\in Z(A)$, the
centraliser of $A$. Given such a $K$ we obtain
\begin{eqnarray*}
\Gamma *_{\langle A\rangle} (K\Gamma'K^{-1}) & = & \langle B,C,KB'K^{-1},KC'K^{-1}\,:\, CB(KC'K^{-1})(KB'K^{-1})=I\rangle, \\
\Gamma *_{\langle DK\rangle} & = & \langle A,C,DK\,:\,C[DK,A^{-1}]=I\rangle.
\end{eqnarray*}
We now explain how to parameterise $Z(A)$.
Since we assumed that $A$ is strongly loxodromic, it has three distinct eigenvalues and hence has three
(complex) one dimensional eigenspaces. Thus, its centraliser $Z(A)$ consists of all elements of ${\rm SL}(3,{\mathbb C})$ 
preserving each of these eigenspaces. This space has two complex dimensions and we parametrise using complex
twist-bend and bulge-turn parameters.

It is easiest to define the twist-bend and bulge turn parameters when $A$ is diagonal; see Goldman \cite{Gol}.
Suppose the strongly loxodromic map $A$ has eigenvalues $\lambda_1,\,\lambda_2,\,\lambda_3$ such that 
$|\lambda_1|> |\lambda_2|> |\lambda_3|$.  Let ${\bf v}_+(A)$, ${\bf v}_0(A)$ and ${\bf v}_-(A)$ be eigenvectors
associated to $\lambda_1$, $\lambda_2$, $\lambda_3$ respectively. 

Conjugating if necessary, assume that ${\bf v}_+(A),\,{\bf v}_0(A),\,{\bf v}_-(A)$ are the standard basis vectors, and so
$A$ is a diagonal matrix
$$
	A=\left(\begin{matrix} \lambda_1 & 0 & 0 \\ 0 & \lambda_2 & 0 \\ 0 & 0 & \lambda_3		
	\end{matrix}\right).
$$
Clearly the centraliser $Z(A)$ of $A$ is the set of all diagonal matrices in $SL(3,\mathbb{C})$. This is the direct product of the 
two one-parameters subgroups
\begin{equation}
	T^u=\left(\begin{array}{ccc}
		e^{u} & 0 & 0 \\
		0 & 1 & 0 \\
		0 & 0 & e^{-u}
	\end{array}\right), \quad 
	U^v=\left(\begin{array}{ccc}
		e^{-v} & 0 & 0 \\
		0 & e^{2v} & 0 \\
		0 & 0 & e^{-v}
	\end{array}\right), \label{eq:tbbt}
\end{equation}
where $u,\,v\in\mathbb{C}$. Write $K\in Z(A)$ as 
$$
K=\left(\begin{matrix} \kappa_1 & 0 & 0 \\ 0 & \kappa_2 & 0 \\ 0 & 0 & \kappa_3 \end{matrix}\right).
$$
If $K=T^uU^v$ then 
$$
\kappa_1=e^{u-v},\quad \kappa_2=e^{2v},\quad \kappa_3=e^{-u-v}.
$$
That is, we can define $u$ and $v$ in a conjugation invariant way as follows. 
We will include the dependence on $\Gamma$ and $\alpha$ which we use later.

\begin{definition}\label{def-t-b-b-t}
First, the {\em bulge-turn} parameter 
$v=(s+i\phi)_\Gamma(\alpha)$ is defined by 
$$
\kappa_2=e^{2v}=e^{2(s+i\phi)_\Gamma(\alpha)},
$$ 
which is the eigenvalue of $K$ corresponding to the eigenvector ${\bf v}_0(A)$. 
In order to make $v$ well defined, we suppose $\phi_\Gamma(\alpha)\in{\mathbb R}/\pi{\mathbb Z}$. Next, we define
the {\em twist-bend} $u=(t+i\theta)_\Gamma(\alpha)$ by 
$$
\kappa_1=e^{u-v}=e^{(t+i\theta)_\Gamma(\alpha)-(s+i\phi)_\Gamma(\alpha)}
$$
which is the eigenvalue of $K$ corresponding to the eigenvector ${\bf v}_+(A)$. In order to make $u$ well defined, 
we suppose $\theta_\Gamma(\alpha)\in{\mathbb R}/2\pi{\mathbb Z}$. 
\end{definition}

For clarity, we have
\begin{enumerate}
\item[(i)] the {\em twist} along $\alpha$ is $t_\Gamma(\alpha)={\rm Re}(u)\in{\mathbb R}$, 
\item[(ii)] the {\em bend} along $\alpha$ is $\theta_\Gamma(\alpha)={\rm Im}(u)\in {\mathbb R}/2\pi{\mathbb R}$,
\item[(iii)] the {\em bulge} along $\alpha$ is $s_\Gamma(\alpha)={\rm Re}(v)\in{\mathbb R}$,
\item[(iv)] the {\em turn} along $\alpha$ is $\phi_\Gamma(\alpha)={\rm Im}(v)\in{\mathbb R}/\pi{\mathbb R}$.
\end{enumerate}

Using the decomposition of $S_g$ along ${\mathcal L}=\{[\gamma_1],\,\ldots,\,[\gamma_{3g-3}]\}$ into three holed spheres
$\{Y_1,\,\ldots,\,Y_{2g-2}\}$ these two operations allow us to construct a representation of $\pi_1(S_g)$ to
${\rm SL}(3,{\mathbb C})$. This yields the following parameters:
\begin{enumerate}
\item[(1)] $6g-6$ complex trace parameters arising from the curves $\gamma_1,\,\ldots,\,\gamma_{3g-3}$, namely 
${\rm tr}(A_1),\,\ldots,\,{\rm tr}(A_{3g-3})$ and 
${\rm tr}(A_1^{-1}),\,\ldots,\,{\rm tr}(A_{3g-3}^{-1})$, where $A_j=\rho\bigl([\gamma_j]\bigr)$, 
\item[(2)] $4g-4$ complex shape parameters arising from the pairs of pants $Y_1,\,\ldots,\, Y_{2g-2}$, namely
$\sigma_+(Y_1),\,\ldots,\,\sigma_+(Y_{2g-2})$ and 
$\sigma_-(Y_1),\,\ldots,\,\sigma_-(Y_{2g-2})$, 
\item[(3)] choices of a root of the for each of of the $2g-2$ polynomials $Q(Y_1),\,\ldots,\,Q(Y_{2g-2})$, 
\item[(4)] $3g-3$ complex twist-bend parameters 
$(t+i\theta)_\Gamma(\gamma_1),\,\ldots,\,(t+i\theta)_\Gamma(\gamma_{3g-3})$ 
and $3g-3$ complex bulge-turn parameters 
$(s+i\phi)_\Gamma(\gamma_1),\,\ldots,\,(s+i\phi)_\Gamma(\gamma_{3g-3})$.
\end{enumerate}

This proves Theorem~\ref{thm:main-SL(3,C)}.

\section{${\rm SL}(2,\mathbb{K})$ coordinates}\label{sec:SL(2,K)}

In this section, we consider representations of $\pi_1(S_g)$ to ${\rm SL}(3,{\mathbb C})$ that factor through 
${\rm SL}(2,{\mathbb K})$ where ${\mathbb K}$ is either
${\mathbb R}$ or ${\mathbb C}$. Most of the construction works in both cases and so it is convenient to cover them
together. We will highlight the places where there is a difference. We are most interested in the case where
the inclusion of ${\rm SL}(2,{\mathbb K})$ in ${\rm SL}(3,{\mathbb C})$ is via the irreducible representation. 
It will be useful to also briefly consider a particular type of reducible representation. 

\subsection{Representations where ${\rm tr}(A)={\rm tr}(A^{-1})$}

It is well known that if $A\in{\rm SL}(2,{\mathbb K})$ then ${\rm tr}(A^{-1})={\rm tr}(A)$. This will also be true
of the images of ${\rm SL}(2,{\mathbb K})$ in ${\rm SL}(3,{\mathbb C})$ we consider. The following lemma
is a simple consequence of this fact.

\begin{lemma}\label{lem-1-evalue}
Suppose that $A$ is an element of ${\rm SL}(3,{\mathbb C})$  for which ${\rm tr}(A^{-1})={\rm tr}(A)$. Then
$A$ has $1$ as an eigenvalue.
\end{lemma}

\begin{proof}
Using Lemma~\ref{lem:char-poly} we see that the characteristc polynomial of $A$ is
\begin{eqnarray*}
\chi_A(x) & = & x^3-{\rm tr}(A)x^2+{\rm tr}(A)x-1 \\
& = & (x-1)\bigl(x^2 -({\rm tr}(A)-1)x+1\bigr).
\end{eqnarray*}
The result follows.
\end{proof}

\medskip

Consider a subgroup $\Gamma=\langle A,B,C:CBA=I\rangle$ of ${\rm SL}(3,{\mathbb C})$ where all elements
have trace equal to the trace of their inverse. This means we must be in the branching locus of the quadratic
$Q$ given by \eqref{eqn:Q-y} whose roots are ${\rm tr}[A,B]$ and ${\rm tr}[B,A]={\rm tr}\bigl([A,B]^{-1}\bigr)$.
The following proposition shows that, in such a group, the equation of the branching locus factorises. In subsequent
sections we will characterise the different factors.

\begin{theorem}\label{thm:branch-fact}
Suppose that $A,B,C\in{\rm SL}(3,{\mathbb C})$ with $CBA=I$ satisfy
$$
\begin{array}{lll}
{\rm tr}(A)={\rm tr}(A^{-1}),\quad & {\rm tr}(B)={\rm tr}(B^{-1}),\quad &{\rm tr}(C)={\rm tr}(C^{-1}),\\
\sigma_+(A,B,C)=\sigma_-(A,B,C),\quad & {\rm tr}[A,B]={\rm tr}\bigl([A,B]^{-1}\bigr). &
\end{array}
$$
where $\sigma_+,\sigma_-$ are given by \eqref{eq:shape+} and \eqref{eq:shape-}. 
Write $a={\rm tr}(A)$, $b={\rm tr}(B)$, $c={\rm tr}(C)$. Then 
\begin{enumerate}
\item[(1)] either $\sigma_+=\sigma_- = 3-a-b-c$
and
$$
{\rm tr}[A,B] = -abc+a^2+b^2+c^2+ab+bc+ac -3a-3b-3c+3.
$$
\item[(2)] or $\sigma_+=\sigma_-$ is a root of the polynomial
\begin{eqnarray*}
T_2(t) & = & t^2-2(a+b+c+1)t \\
&& \quad -4abc+a^2+b^2+c^2-2ab-2bc-2ac-2a-2b-2c-3.
\end{eqnarray*}
and
$$
{\rm tr}[A,B] = (a+b+c+1)\sigma_++(a+1)(b+1)(c+1)-1.
$$
\end{enumerate}
\end{theorem}

\begin{proof}
Setting $y_1=y_5=a$, $y_2=y_6=b$, $y_3=y_7=c$ and $y_4=y_8=t$ in the polynomials $S$ and $P$ 
from \eqref{eq:S-y} and \eqref{eq:P-y} gives
\begin{eqnarray*}
S & = & t^2-2abc+a^2+b^2+c^2-3, \\
P & = & a^2b^2c^2 +2a^2b^2c+2a^2bc^2+2ab^2c^2 +a^2b^2+b^2c^2+a^2c^2  \\
&& \quad -4abc^2-4ab^2c-4a^2bc +2a^3+2b^3+2c^3 +6abc \\
&& \quad + 2a^2bct+2abc^2t+2ab^2ct + 2ab^2t+2a^2bt+2a^2ct+2ac^2t+2bc^2t+2b^2ct \\
&& \quad + 2(ab+bc+ac)(t^2-3t) + (a^2+b^2+c^2-6)t^2+2t^3+9.
\end{eqnarray*}
Since ${\rm tr}[A,B]={\rm tr}[B,A]$ the polynomial $Q$ from \eqref{eqn:Q-y} has repeated roots. This means 
that $0=S^2-4P$. Substituting from the above expressions we find
\begin{eqnarray*}
\lefteqn{S^2-4P} \\ 
& = & (t^2-2abc+a^2+b^2+c^2-3)^2 \\
&& \quad -4a^2b^2c^2 -8a^2b^2c-8a^2bc^2-8ab^2c^2 -4a^2b^2-4b^2c^2-4a^2c^2  \\
&& \quad +16abc^2+16ab^2c+16a^2bc -8a^3-8b^3-8c^3 -24abc \\
&& \quad -8a^2bct-8abc^2t-8ab^2ct \\
&&\quad -8ab^2t-8a^2bt-8a^2ct-8ac^2t-8bc^2t-8b^2ct \\
&& \quad -8(ab+bc+ac)(t^2-3t) -4(a^2+b^2+c^2-6)t^2+2t^3-36 \\
& = & t^4-8t^3-2(2abc+a^2+b^2+c^2+4ab+4bc+4ac-9)t^2\\ 
&& \quad -8(a^2bc+ab^2c+abc^2+a^2b+ab^2+b^2c+bc^2+a^2c+ac^2-3ab-3bc-3ac)t\\ 
&& \quad -4a^3bc-4ab^3c-4abc^3-8a^2b^2c-8a^2bc^2-8ab^2c^2\\ 
&& \quad +a^4+b^4+c^4-2a^2b^2-2b^2c^2-2a^2c^2+16a^2bc+16ab^2c+16abc^2\\ 
&& \quad -8a^3-8b^3-8c^3+18a^2+18b^2+18c^2-27\\ 
& = & (t+a+b+c-3)^2\\ 
&& \quad \cdot (t^2-2(a+b+c+1)t-4abc+a^2+b^2+c^2-2(ab+bc+ac+a+b+c)-3).\nonumber
\end{eqnarray*}
Therefore the possible values of $t=\sigma_+=\sigma_-$ correspond to the two cases in the statement
of the theorem. Substituting these into ${\rm tr}[A,B]=S/2$ gives the values of ${\rm tr}[A,B]$.
\end{proof}

\medskip

\subsection{Two-generator subgroups of ${\rm SL}(2,{\mathbb C})$}

We will use the following classical theorem of Fricke and Vogt; see Theorem A of Goldman \cite{Gol1}:

\begin{theorem}[Fricke, Vogt] \label{thm-fv}
Let $f:{\rm SL}(2,{\mathbb C})\times {\rm SL}(2,{\mathbb C})\longrightarrow {\mathbb C}$ be a regular function that is
invariant under the action of ${\rm SL}(2,{\mathbb C})$ by conjugation. Then there exists a polynomial function 
$F(x,y,z)\in{\mathbb C}[x,y,z]$ so that
$$
f(A,B)=F\bigl({\rm tr}(A),{\rm tr}(B),{\rm tr}(AB)\bigr).
$$
Furthermore, for all $(x,y,z)\in{\mathbb C}^3$ there exist $A,\,B\in{\rm SL}(2,{\mathbb C})$ so that
$$
{\rm tr}(A)=x,\quad {\rm tr}(B)=y,\quad {\rm tr}(AB)=z.
$$
\end{theorem}

In particular, Fricke and Vogt show we can express ${\rm tr}(A^{-1}B)$ or ${\rm tr}[A,B]$ as the following 
polynomials in ${\rm tr}(A)$, ${\rm tr}(B)$ and ${\rm tr}(AB)$:
\begin{eqnarray}
{\rm tr}(A^{-1}B) & = & {\rm tr}(A){\rm tr}(B)-{\rm tr}(AB), \label{eq-fv1} \\
{\rm tr}[A,B] & = & {\rm tr}^2(A)+{\rm tr}^2(B)+{\rm tr}^2(AB)-2-{\rm tr}(A){\rm tr}(B){\rm tr}(AB). \label{eq-fv2}
\end{eqnarray} 

This theorem almost says that the traces ${\rm tr}(A)$, ${\rm tr}(B)$, ${\rm tr}(AB)$ determine the pair $(A,B)$ up to conjugation.
In fact to get this statement we need to exclude the case where $A$ and $B$ commute.

\begin{proposition}[Section~2.2 of Goldman \cite{Gol1}] Let $A,\,B,\,A',\,B'\in{\rm SL}(2,{\mathbb C})$. Suppose 
$$
{\rm tr}(A)={\rm tr}(A'),\quad {\rm tr}(B)={\rm tr}(B'),\quad {\rm tr}(AB)={\rm tr}(A'B')
$$
and ${\rm tr}[A,B]\neq 2$ (so also ${\rm tr}[A',B']\neq 2$). Then there exists $D\in{\rm SL}(2,{\mathbb C})$ so that
$A'=DAD^{-1}$ and $B'=DBD^{-1}$.
\end{proposition}

In the case where $A$, $B$, $C$ are loxodromic (hyperbolic) elements of ${\rm SL}(2,{\mathbb R})$ satisfying
$CBA=I$ there are various possibilities for the configuration of their axes in the hyperbolic plane. We are interested
in the case where the axes are pairwise disjoint and bound a common region. We can characterise this
configuration using traces.

\begin{proposition}[Gilman and Maskit \cite{G-M}]\label{prop-gm}
Let $A,\,B\,C$ be hyperbolic elements of ${\rm SL}(2,\mathbb{R})$ with $CBA=I$. Suppose that the axes of 
$A$, $B$ and $C$ are pairwise disjoint and that they bound a region in the hyperbolic plane. Then
\begin{equation*}
{\rm tr}(A){\rm tr}(B){\rm tr}(C)<0
\end{equation*}
\end{proposition}

\subsection{Reducible representations}

Suppose that $\widehat{A},\,\widehat{B},\,\widehat{C}$ are elements of ${\rm SL}(2,{\mathbb C})$ with 
$\widehat{C}\widehat{B}\widehat{A}=I$. 
Then we define the following
block diagonal elements of ${\rm SL}(3,{\mathbb C})$:
\begin{equation}\label{eq-red-A-B-C}
A=\left(\begin{matrix} \widehat{A} & 0 \\ 0 & 1 \end{matrix}\right), \quad 
B=\left(\begin{matrix} \widehat{B} & 0 \\ 0 & 1 \end{matrix}\right), \quad 
C=\left(\begin{matrix} \widehat{C} & 0 \\ 0 & 1 \end{matrix}\right).
\end{equation}
Clearly we have $CBA=I$. It is also clear that ${\rm tr}(A^{-1})={\rm tr}(A)$, ${\rm tr}(B^{-1})={\rm tr}(B)$ and 
${\rm tr}(C^{-1})={\rm tr}(C)$. Similarly ${\rm tr}([A,B]^{-1})={\rm tr}[A,B]$.
Note that in this case the traces do not determine the group up to conjugation. In order to see this, observe that if ${\bf a}$
and ${\bf b}$ are any column vectors in ${\mathbb C}^2$ then 
$$
A'=\left(\begin{matrix} \widehat{A} & {\bf a} \\ 0 & 1 \end{matrix}\right), \quad 
B'=\left(\begin{matrix} \widehat{B} & {\bf b} \\ 0 & 1 \end{matrix}\right), \quad 
C'=\left(\begin{matrix} \widehat{C} & -\widehat{C}\widehat{B}{\bf a} -\widehat{C}{\bf b} \\ 0 & 1 \end{matrix}\right)
$$
satisfy $C'B'A'=I$ and ${\rm tr}(A')={\rm tr}(A)$ etc.

Note that there are other reducible representations, for example in \eqref{eq-red-A-B-C}
we can multiply $\widehat{A}$ by $\lambda\in{\mathbb C}-\{0\}$ and in $A$ make the bottom right hand 
entry $\lambda^{-2}$ instead of $1$. 
Similarly for $B$ and $C$. Such representations do not satisfy ${\rm tr}(A^{-1})={\rm tr}(A)$, and we will not consider them here. 

It is straightforward to use equations \eqref{eq-fv1} to write ${\rm tr}(A^{-1}B)$, and hence the shape invariant 
$\sigma_+=\sigma_-$ in terms of ${\rm tr}(A)$, ${\rm tr}(B)$ and ${\rm tr}(C)$.

\begin{lemma}\label{lem-shape-red}
Let $A,\,B,\,C\in{\rm SL}(3,{\mathbb C})$ with $CBA=I$ be as given in \eqref{eq-red-A-B-C}. Let
$\sigma_+$ and $\sigma_-$ be given by \eqref{eq:shape+} and  \eqref{eq:shape-}. Then
\begin{equation}\label{eq-red-shape}
\sigma_+=\sigma_-=3-{\rm tr}(A)-{\rm tr}(B)-{\rm tr}(C).
\end{equation}
\end{lemma}

\begin{proof}
First observe that ${\rm tr}(A)={\rm tr}(\widehat{A})+1$ and so on. Therefore, using \eqref{eq-fv1} we have
\begin{eqnarray*}
\sigma_+ & = & {\rm tr}(A^{-1}B)-{\rm tr}(A^{-1}){\rm tr}(B) \\
& = & \bigl({\rm tr}(\widehat{A}^{-1}\widehat{B})+1\bigr)-\bigl({\rm tr}(\widehat{A})+1\bigr)\bigl({\rm tr}(\widehat{B})+1\bigr) \\
& = & {\rm tr}(\widehat{A}){\rm tr}(\widehat{B})-{\rm tr}(\widehat{C})+1
-{\rm tr}(\widehat{A}){\rm tr}(\widehat{B})-{\rm tr}(\widehat{A})-{\rm tr}(\widehat{B})-1 \\
& = & 3-{\rm tr}(A)-{\rm tr}(B)-{\rm tr}(C).
\end{eqnarray*}
Since each of the traces of $A$, $B$ and $A^{-1}B$ equals the trace of its inverse we have $\sigma_-=\sigma_+$.
\end{proof}

\medskip

\begin{theorem}\label{thm-red}
Let $A,\,B,\,C$ be any elements of ${\rm SL}(3,{\mathbb C})$ satisfying:
\begin{enumerate}
\item[(a)] $CBA=I$, 
\item[(b)] ${\rm tr}(A^{-1})={\rm tr}(A)$, ${\rm tr}(B^{-1})={\rm tr}(B)$, ${\rm tr}(C^{-1})={\rm tr}(C)$ and 
${\rm tr}([A,B]^{-1})={\rm tr}[A,B]$,
\item[(c)] $\sigma_+(A,B,C)=\sigma_-(A,B,C)$.
\end{enumerate}
If $\sigma_+(A,B,C)=3-{\rm tr}(A)-{\rm tr}(B)-{\rm tr}(C)$ then the group 
$\langle A,B,C:CBA=I\rangle$ is reducible.
\end{theorem}

\begin{proof}
Define $a:={\rm tr}(A)$, $b:={\rm tr}(B)$, $c:={\rm tr}(C)$. 

Using the theorem of Fricke and Vogt, Theorem~\ref{thm-fv} we can find 
$\widehat{A},\,\widehat{B},\,\widehat{C}\in{\rm SL}(2,{\mathbb C})$ with $\widehat{C}\widehat{B}\widehat{A}=I$ 
and so that ${\rm tr}(\widehat{A})=a-1$, ${\rm tr}(\widehat{B})=b-1$, ${\rm tr}(\widehat{C})=c-1$. 
Thus, we can construct 
$A_0$, $B_0$, $C_0$ in ${\rm SL}(3,{\mathbb C})$ of the form \eqref{eq-red-A-B-C} with 
${\rm tr}(A_0)=a$, ${\rm tr}(B_0)=b$ and ${\rm tr}(C_0)=c$.  
Using Lemma~\ref{lem-shape-red} we have
$$
(\sigma_0)_+ =(\sigma_0)_- = 3-{\rm tr}(A_0)-{\rm tr}(B_0)-{\rm tr}(C_0) \\
$$
Therefore, there is a reducible representation $\langle A_0,B_0,C_0:C_0B_0A_0=I\rangle$ for which the eight traces 
agree with those of $\langle A,B,C:CBA=I\rangle$. Hence, the latter group must also be reducible (see 
Lawton's theorem, Theorem~\ref{thm-lawton} (3), and the remark following this theorem). 
\end{proof}

\medskip 

\subsection{Irreducible representations}\label{sec-irred}

In this section, we consider the irreducible representation of ${\rm SL}(2,{\mathbb K})$ to ${\rm SL}(3,{\mathbb C})$ for
${\mathbb K}={\mathbb R}$ or ${\mathbb C}$.  

Consider the following map from ${\mathbb K}^2$ to ${\mathbb K}^3$:
$$
\Phi:{\bf w}=\left(\begin{matrix} w_1 \\ w_2 \end{matrix}\right) \longmapsto 
\left(\begin{matrix} -w_1^2 \\ \sqrt{2}\,w_1w_2 \\ w_2^2 \end{matrix}\right).
$$
Writing $z_1=-w_1^2$, $z_2=\sqrt{2}\,w_1w_2$ and $z_3=w_2^2$ we see that the image of of $\Phi$ satisfies
$2z_1z_3+z_2^2=0$. We can write the latter equation in terms of a quadratic form.
Let 
\begin{equation}\label{eq:J}
J=\left(\begin{matrix} 0 & 0 & 1 \\ 0 & 1 & 0 \\ 1 & 0 & 0 \end{matrix}\right).
\end{equation}
Then we can write
$$
2z_1z_3+z_2^2=(\begin{matrix} z_1 & z_2 & z_3 \end{matrix})
\left(\begin{matrix} 0 & 0 & 1 \\ 0 & 1 & 0 \\ 1 & 0 & 0 \end{matrix}\right)
\left(\begin{matrix} z_1 \\ z_2 \\ z_3 \end{matrix}\right) ={\bf z}^t J{\bf z}.
$$
Therefore $\Phi({\bf w})$ lies in the zero set of the quadratic form defined by $J$. It is easy to check $J$ has 
signature $(2,1)$, that is it has two positive eigenvalues (both $+1$) and one negative eigenvalue (which is $-1$).

Now consider ${\rm SL}(2,{\mathbb K})$. It acts naturally by left multiplication on ${\mathbb K}^2$. Applying $\Phi$
enables to to construct a map $\Phi_*:{\rm SL}(2,{\mathbb K}) \longmapsto {\rm SL}(3,{\mathbb K})$ as follows
$$
\Phi_*(\widehat{A})\Phi({\bf w}) = \Phi(\widehat{A}{\bf w}).
$$
A short calculation gives
\begin{equation}\label{eq-Phi*}
\Phi_*:\left(\begin{matrix} a & b \\ c & d \end{matrix}\right) \longmapsto
\left(\begin{matrix} a^2 & -\sqrt{2}\, ab & -b^2 \\
-\sqrt{2}\,ac & ad+bc & \sqrt{2}\,bd \\
-c^2 & \sqrt{2}\, cd & d^2
\end{matrix}\right).
\end{equation}
It is not hard to check that $\Phi_*$ is a homomorphism whose kernel is $\{\pm I\}$
and whose image is contained in ${\rm SO}(J;{\mathbb K})$. In fact, it is not hard to check that
when ${\mathbb K}={\mathbb C}$ then $\Phi_*$ maps ${\rm SL}(2,{\mathbb C})$ onto
${\rm SO}(J;{\mathbb C})$ and when ${\mathbb K}={\mathbb R}$ then $\Phi_*$ maps ${\rm SL}(2,{\mathbb R})$ onto
the identity component ${\rm SO}_0(J;{\mathbb R})$ of ${\rm SO}(J;{\mathbb R})$. Note that since $J$ has signature $(2,1)$, 
this means $\Phi_*$ is a bijection between ${\rm SO}_0(2,1;{\mathbb R})$ and 
${\rm PSL}(2,{\mathbb R})={\rm SL}(2,{\mathbb R})/\{\pm I\}$.
When ${\mathbb K}={\mathbb C}$ the signature of $J$ is not well defined. 

\begin{lemma}\label{lem-eval-irred}
Let $\widehat{A}\in{\rm SL}(2,{\mathbb K})$ with eigenvalues $\lambda$ and $\lambda^{-1}$. 
Then the eigenvalues of $A=\Phi_*(\widehat{A})$ are $\lambda^2$, $1$ and $\lambda^{-2}$. 
In particular, ${\rm tr}(A)={\rm tr}^2(\widehat{A})-1$. Moreover,
if ${\bf u}\in{\mathbb K}^2$ is an eigenvector of $\widehat{A}$ with eigenvalue $\lambda$, 
then $\Phi({\bf u})$ is an eigenvector of $A$ with eigenvalue $\lambda^2$.
\end{lemma}

\begin{proof}
It is not hard to see from the \eqref{eq-Phi*} that ${\rm tr}(A^{-1})={\rm tr}(A)$. From Lemma~\ref{lem-1-evalue} we see that
$A$ has $1$ as an eigenvalue. Now suppose 
${\bf u}\in{\mathbb K}^2$ is an eigenvector of $\widehat{A}$ with eigenvalue $\lambda$. 
Then
$$
A\Phi({\bf u})=\Phi_*(\widehat{A})\Phi({\bf u})=\Phi(\widehat{A}{\bf u})=\Phi(\lambda{\bf u})=\lambda^2\Phi({\bf u}).
$$
This gives the second part. Thus, the eigenvalues are $\lambda^2$, $1$ and $\lambda^{-2}$.
Therefore 
$$
{\rm tr}(A)=\lambda^2+1+\lambda^{-2}=(\lambda+\lambda^{-1})^2-1={\rm tr}^2(\widehat{A})-1.
$$
\end{proof}

\medskip
 
We now consider $\widehat{A}$, $\widehat{B}$ and $\widehat{C}$ in ${\rm SL}(2,{\mathbb K})$ with 
$\widehat{C}\widehat{B}\widehat{A}=I$ 
and the corresponding $A=\Phi_*(\widehat{A})$, $B=\Phi_*(\widehat{B})$, $C=\Phi_*(\widehat{C})$ in
${\rm SL}(3,{\mathbb C})$. Since $\Phi_*$ is a homomorphism we have $CBA=I$.  We know that 
$\langle \widehat{A},\, \widehat{B},\, \widehat{C} : \widehat{C}\widehat{B}\widehat{A}=I\rangle <{\rm SL}(2,{\mathbb K})$
is determined up to conjugation by ${\rm tr}(\widehat{A})$,  ${\rm tr}(\widehat{B})$ and ${\rm tr}(\widehat{C})$ using 
the theorem of Fricke and Vogt, Theorem~\ref{thm-fv}. Therefore, it is tempting to say that its image under $\Phi_*$, 
namely $\langle A,\,B,\,C : CBA=I\rangle$, is determined up to conjugation by ${\rm tr}(A)$, ${\rm tr}(B)$ and ${\rm tr}(C)$. 
However, this is not quite true. Consider $\widehat{A}_1$, $\widehat{B}_1$, $\widehat{C}_1$ in ${\rm SL}(2,{\mathbb K})$
with $\widehat{C}_1\widehat{B}_1\widehat{C}_1=I$ and ${\rm tr}(\widehat{A}_1)=-{\rm tr}(\widehat{A})$, 
${\rm tr}(\widehat{B}_1)=-{\rm tr}(\widehat{B})$, ${\rm tr}(\widehat{C}_1)=-{\rm tr}(\widehat{C})$ (such matrices exist by
the theorem of Fricke and Vogt). As we have already seen, 
${\rm tr}(\widehat{A}){\rm tr}(\widehat{B}){\rm tr}(\widehat{C})$ is independent of the choice of lift of $\widehat{A}$ and 
$\widehat{B}$ from ${\rm PSL}(2,{\mathbb K})$ to ${\rm SL}(2,{\mathbb K})$. Since
$$
{\rm tr}(\widehat{A}_1){\rm tr}(\widehat{B}_1){\rm tr}(\widehat{C}_1)
=-{\rm tr}(\widehat{A}){\rm tr}(\widehat{B}){\rm tr}(\widehat{C}),
$$
the two groups
$\langle \widehat{A},\widehat{B},\widehat{C} :\widehat{C}\widehat{B}\widehat{A}=I\rangle$ \
$\langle \widehat{A}_1,\widehat{B}_1,\widehat{C}_1 :\widehat{C}_1\widehat{B}_1\widehat{A}_1=I\rangle$ \
correspond to different subgroups of ${\rm PSL}(2,{\mathbb K})$. 

Write $A_1=\Phi_*(\widehat{A}_1)$, $B_1=\Phi_*(\widehat{B}_1)$, $C_1=\Phi_*(\widehat{CB}_1)$. Then 
$$
{\rm tr}(A_1)=\bigl({\rm tr}(\widehat{A}_1)\bigr)^2-1
=\bigl(-{\rm tr}(\widehat{A})\bigr)^2-1={\rm tr}(A),
$$
and similarly ${\rm tr}(B_1)={\rm tr}(B)$ and ${\rm tr}(C_1)={\rm tr}(C)$. However 
$\langle A,B,C : CBA=I\rangle$ and $\langle A_1,B_1,C_1 : C_1B_1A_1=I\rangle$ are not conjugate.
The ambiguity is captured by
looking at ${\rm tr}(A^{-1}B)$ and ${\rm tr}(A_1^{-1}B_1)$, or in a more invariant way by the shape invariants
$\sigma_+=\sigma_-={\rm tr}(A^{-1}B)-{\rm tr}(A^{-1}){\rm tr}(B)$ and 
$(\sigma_1)_+=(\sigma_1)_-={\rm tr}(A_1^{-!}B_1)-{\rm tr}(A_1^{-1}){\rm tr}(B_1)$ which are the two roots of the 
quadratic polynomial in the following lemma.

\begin{lemma}\label{lem-shape-irred}
Suppose that $A$, $B$ and $C$ are all in the image of $\Phi_*$ and satisfy $CBA=I$. Write 
$a={\rm tr}(A)$, $b={\rm tr}(B)$ and $c={\rm tr}(C)$. Let 
$\sigma_+$ and $\sigma_-$ be given by \eqref{eq:shape+} and  \eqref{eq:shape-}. Then $\sigma_+=\sigma_-$
is a  root of the polynomial
$$
X^2-2\bigl(a+b+c+1)X-4abc+a^2+b^2+c^2-2ab-2ac-2bc-2a-2b-2c-3.
$$
That is,
$$
X=a+b+c+1\pm 2\sqrt{(a+1)(b+1)(c+1)}.
$$
\end{lemma}

\begin{proof}
Writing $\widehat{A}$, $\widehat{B}$ for matrices that are sent to $A$ and $B$ by $\Phi_*$, we have
\begin{eqnarray*}
{\rm tr}(A^{-1}B) & = & {\rm tr}^2(\widehat{A}^{-1}\widehat{B})-1 \\
& = & \Bigl({\rm tr}(\widehat{A}\widehat{B})-{\rm tr}(\widehat{A}){\rm tr}(\widehat{B})\Bigr)^2-1 \\
& = & \bigl({\rm tr}(AB)+1)-2{\rm tr}(\widehat{A}){\rm tr}(\widehat{B}){\rm tr}(\widehat{A}\widehat{B})
+\bigl({\rm tr}(A)+1\bigr)\bigl({\rm tr}(B)+1\bigr) -1 \\
& = & {\rm tr}(A){\rm tr}(B) +{\rm tr}(A)+{\rm tr}(B)+{\rm tr}(C)+1
-2{\rm tr}(\widehat{A}){\rm tr}(\widehat{B}){\rm tr}(\widehat{A}\widehat{B}).
\end{eqnarray*}
Thus
\begin{equation}\label{eq-shape-irred}
\sigma_+=\sigma_-={\rm tr}(A)+{\rm tr}(B)+{\rm tr}(C)+1
-2{\rm tr}(\widehat{A}){\rm tr}(\widehat{B}){\rm tr}(\widehat{A}\widehat{B}).
\end{equation}
Therefore
\begin{eqnarray*}
\lefteqn{
\bigl(\sigma_\pm -{\rm tr}(A)-{\rm tr}(B)-{\rm tr}(C)-1\bigr)^2 } \\
& = & 4{\rm tr}^2(\widehat{A}){\rm tr}^2(\widehat{B}){\rm tr}^2(\widehat{A}\widehat{B}) \\
& = & 4\bigl({\rm tr}(A)+1\bigr)\bigl({\rm tr}(B)+1\bigr)\bigl({\rm tr}(C)+1\bigr).
\end{eqnarray*}
The result follows by rearranging this expression.
\end{proof}

\medskip

\begin{corollary}\label{cor-irred}
Suppose that $A,B,C\in{\rm SL}(3,{\mathbb C})$ with $CBA=I$. Suppose that 
$$
\begin{array}{lll}
{\rm tr}(A)={\rm tr}(A^{-1}),\quad &{\rm tr}(B)={\rm tr}(B^{-1}),\quad &{\rm tr}(C)={\rm tr}(C^{-1}),\\
\sigma_+(A,B,C)=\sigma_-(A,B,C),\quad &{\rm tr}[A,B]={\rm tr}\bigl([A,B]^{-1}\bigr). &
\end{array}
$$
where $\sigma_+,\sigma_-$ are given by \eqref{eq:shape+} and \eqref{eq:shape-}. 
Then either $\langle A,B,C:CBA=I\rangle$ is reducible or else, up to conjugacy, it is in the image of the map 
$\Phi_*$ from \eqref{eq-Phi*}.
\end{corollary}

\begin{proof}
From Theorem \ref{thm:branch-fact} either the traces of $A$, $B$, $C$ and the shape invariants satisfy 
$\sigma_+=\sigma_- = 3-{\rm tr}(A)-{\rm tr}(B)-{\rm tr}(C)$, in which case $\langle A,B,C:CBA=I\rangle$ is reducible by 
Theorem \ref{thm-red}, or else they satisfy 
\begin{eqnarray*}
0 & = & t^2-2(a+b+c+1)t \\
&& \quad -4abc+a^2+b^2+c^2-2ab-2bc-2ac-2a-2b-2c-3
\end{eqnarray*}
where $a={\rm tr}(A)$, $b={\rm tr}(B)$, $c={\rm tr}(C)$ and $t=\sigma_+(A,B,C)$. In this case there are
matrices $\widehat{A}$, $\widehat{B}$ and $\widehat{C}$ in ${\rm SL}(2,{\mathbb C})$ whose images 
under $\Phi_*$ have the desired traces. Providing this representation is irreducible, it is determined by these
traces up to conjugation, see Theorem~\ref{thm-lawton} (3). This gives the result.
\end{proof}

\medskip

In our application to three holed spheres, there is a consistent choice of root of the equation 
from Lemma~\ref{lem-shape-irred}. Let $Y$ be a a three holed sphere with boundary curves 
$\alpha$, $\beta$ and $\gamma$. We are interested in Fuchsian representations of $\pi_1(Y)$ in the case where
${\mathbb K}={\mathbb R}$ and quasi-Fuchsian representations in the case where ${\mathbb K}={\mathbb C}$. In the first
case, these are representations $\rho:\pi_1(Y)\longmapsto \Gamma$, where $\Gamma$ is a subgroup of 
${\rm Isom}_+({\bf H}^2_{\mathbb R})$, the orientation preserving isometries
of the hyperbolic plane, with the property that ${\bf H}^2_{\mathbb R}/\Gamma$ is homeomorphic to $Y$.
Necessarily this means that $\alpha$, $\beta$ and $\gamma$ are represented by hyperbolic (loxodromic) maps.

\begin{proposition}\label{prop-shape-fuchsian}
Let $Y$ be a three holed sphere with boundary curves $\alpha$, $\beta$, $\gamma$ and let 
$\rho:\pi_1(Y)\longrightarrow \Gamma<{\rm SO}_0(2,1)$ be a Fuchsian representation of its fundamental group. 
Let $A=\rho([\alpha])$, $B=\rho([\beta])$ and $C=\rho([\gamma])$. Then the shape invariants $\sigma_{\pm}$ of 
$\Gamma$ and the trace of $[A,B]$ are given by
\begin{eqnarray*}
\sigma_+ & = & \sigma_-={\rm tr}(A)+{\rm tr}(B)+{\rm tr}(C)+1
+2\sqrt{\bigl({\rm tr}(A)+1\bigr)\bigl({\rm tr}(B)+1\bigr)\bigl({\rm tr}(C)+1\bigr)}, \\
{\rm tr}[A,B] & = & \Bigl({\rm tr}(A)+{\rm tr}(B)+{\rm tr}(C)+1+
\sqrt{\bigl({\rm tr}(A)+1\bigr)\bigl({\rm tr}(B)+1\bigr)\bigl({\rm tr}(C)+1\bigr)}\Bigr)^2-1
\end{eqnarray*}
where we take the positive square root.
\end{proposition}

\begin{proof}
By construction, there exist $\widehat{A},\,\widehat{B}\in{\rm SL}(2,{\mathbb R})$ so that 
$A=\Psi_*(\widehat{A})$, $B=\Psi_*(\widehat{B})$, $C=(BA)^{-1}=\Psi_*(\widehat{A}^{-1}\widehat{B}^{-1})$, 
Hence
$$
{\rm tr}(A)+1={\rm tr}^2(\widehat{A}),\quad {\rm tr}(B)+1={\rm tr}^2(\widehat{B}),\quad 
{\rm tr}(C)+1={\rm tr}^2(\widehat{A}\widehat{B}).
$$
Using Proposition~\ref{prop-gm} we have
$$
{\rm tr}(\widehat{A}){\rm tr}(\widehat{B}){\rm tr}(\widehat{A}\widehat{B})<0.
$$
Therefore, taking the positive square root, we have
$$
{\rm tr}(\widehat{A}){\rm tr}(\widehat{B}){\rm tr}(\widehat{A}\widehat{B})
=-\sqrt{\bigl({\rm tr}(A)+1\bigr)\bigl({\rm tr}(B)+1\bigr)\bigl({\rm tr}(C)+1\bigr)}.
$$
We obtain the result by substituting this into equation \eqref{eq-shape-irred}.
\end{proof}

\medskip

The space of quasi-Fuchsian representations of $\pi_1(Y)$ is a connected set that contains the Fuchsian representations
and on which $A$, $B$ and $C$ are always loxodromic. A consequence of the latter condition is that for all
quasi-Fuchsian representations ${\rm tr}(A)\neq -1$, ${\rm tr}(B)\neq -1$ and ${\rm tr}(C)\neq -1$,  Thus, on the
space of quasi-Fuchsian representations there
is a well defined branch of $\sqrt{\bigl({\rm tr}(A)+1\bigr)\bigl({\rm tr}(B)+1\bigr)\bigl({\rm tr}(C)+1\bigr)}$ that
agrees with the positive square root when the three traces are real and positive. This branch is obtained
by analytic continuation along paths of quasi-Fuchsian representations.

\begin{corollary}\label{cor-shape-quasiF}
Let $Y$ be a three holed sphere with boundary curves $\alpha$, $\beta$, $\gamma$ and let 
$\rho:\pi_1(Y)\longrightarrow \Gamma<{\rm SO}(3;{\mathbb C})$ be a quasi-Fuchsian representation of its fundamental group. 
Let $A=\rho([\alpha])$, $B=\rho([\beta])$ and $C=\rho([\gamma])$. Then the shape invariants $\sigma_\pm$ of $\Gamma$ 
and the trace of $[A,B]$ are given by
\begin{eqnarray*}
\sigma_+ & = & \sigma_-={\rm tr}(A)+{\rm tr}(B)+{\rm tr}(C)+1
+2\sqrt{\bigl({\rm tr}(A)+1\bigr)\bigl({\rm tr}(B)+1\bigr)\bigl({\rm tr}(C)+1\bigr)}, \\
{\rm tr}[A,B] & = & \Bigl({\rm tr}(A)+{\rm tr}(B)+{\rm tr}(C)+1+
\sqrt{\bigl({\rm tr}(A)+1\bigr)\bigl({\rm tr}(B)+1\bigr)\bigl({\rm tr}(C)+1\bigr)}\Bigr)^2-1
\end{eqnarray*}
where the square root is a well defined branch that agrees with the positive square root when all three traces are
real and positive.
\end{corollary}

We now consider the twist-bend-bulge-turn parameters associated to the loxodromic map $A$. As before, we assume that 
${\bf v}_+(A),\,{\bf v}_0(A),\,{\bf v}_-(A)$ are the standard basis vectors, and so
$A$ is a diagonal matrix
$$
	A=\left(\begin{matrix} \lambda & 0 & 0 \\ 0 & 1 & 0 \\ 0 & 0 & \lambda^{-1}		
	\end{matrix}\right)
$$
where $\lambda\in{\mathbb K}$ has $|\lambda|>1$. 
We then consider $K\in{\rm SO}(J;{\mathbb K})$ in the centraliser $Z(A)$ of $A$. This has the form 
$$
K=\left(\begin{matrix} \kappa_1 & 0 & 0 \\ 0 & \kappa_2 & 0 \\ 0 & 0 & \kappa_3 \end{matrix}\right).
$$
Since $K$ is in the image of $\Phi_*$, we see that $\kappa_2=1$ and $\kappa_3=\kappa_1^{-1}$. Thus
$K=T^u$ for some $u\in{\mathbb K}$. Hence the bulge and the turn are both zero.

Summarising, when ${\mathbb K}={\mathbb R}$ a representation of $\pi_1(S_g)$ 
in ${\mathcal T}(S_g)$
is determined by the following parameters
\begin{enumerate}
\item[(1)] $3g-3$ real trace parameters ${\rm tr}(A_1),\,\ldots,\,{\rm tr}(A_{3g-3})$,
\item[(2)] $3g-3$ real twist parameters $t_\Gamma(\gamma_1),\,\ldots,\,t_\Gamma(\gamma_{3g-3})$.
\end{enumerate}
Moreover, the shape invariants are determined by the trace parameters using equation \eqref{eq:shape-SL} and the 
commutator polynomials $Q(Y_1),\,\ldots,\,Q(Y_{2g-2})$ all have repeated roots. Also the real bend parameters
$\theta_\Gamma(\gamma_1),\ldots,\,\theta_\Gamma(\gamma_{3g-3})$ and the complex bulge-turn 
turn parameters $(s+i\phi)_\Gamma(\gamma_1),\ldots,\,(s+i\phi)_\Gamma(\gamma_{3g-3})$ are all zero. 
This proves Theorem~\ref{thm:main-SO(2,1)}.

Likewise, when ${\mathbb K}={\mathbb C}$, a representation $\pi_1(S_g)$ 
in ${\mathcal Q}(S,{\mathcal L},{\rm SO}(3;{\mathbb C})$ is determined by the following parameters
\begin{enumerate}
\item[(1)] $3g-3$ complex trace parameters ${\rm tr}(A_1),\,\ldots,\,{\rm tr}(A_{3g-3})$,
\item[(2)] $3g-3$ complex twist-bend parameters $(t+i\theta)_\Gamma(\gamma_1),\,\ldots,\,(t+i\theta)_\Gamma(\gamma_{3g-3})$.
\end{enumerate}
Moreover, the shape invariants are determined by the trace parameters using equation \eqref{eq:shape-SL} and the 
commutator polynomials $Q(Y_1),\,\ldots,\,Q(Y_{2g-2})$ all have repeated roots. 
Finally, the complex bulge-turn turn parameters $(s+i\phi)_\Gamma(\gamma_1),\ldots,\,(s+i\phi)_\Gamma(\gamma_{3g-3})$ 
are all zero. 
This proves Theorem~\ref{thm:main-SO(3,C)}.

Finally, consider an irreducible subgroup of $\Gamma=\langle A,B,C : CBA=I\rangle$ where the traces $A$, $B$ and $C$
all equal the traces of their respective inverses and also $\sigma_+=\sigma_-$, so ${\rm tr}(A^{-1}B)={\rm tr}(B^{-1}A)$
and where $Q$ has repeated roots, so ${\rm tr}[A,B]={\rm tr}([A,B]^{-1})$. Using Theorem~\ref{thm:branch-fact} we see that
either $\sigma_+=3-{\rm tr}(A)-{\rm tr}(B)-{\rm tr}(C)$ or else it is a root of a particular quadratic polynomial 
$T_2$. Since the group is assumed to be irreducible, the former cannot happen, using Theorem~\ref{thm-red}. 
Hence the traces must satisfy $T_2$. Using Lawton's theorem, $\Gamma$ is uniquely determined up to conjugation by 
these traces and by Lemma~\ref{lem-shape-irred} there is a representation in the image of $\Phi_*$ with
the same values for the traces. Hence $\Gamma$ is a subgroup of ${\rm SO}(3;{\mathbb C})$. This proves
Theorem~\ref{thm-pants} (1).

\section{${\rm SL}(3,\mathbb{R})$-coordinates}\label{sec:SL(3,R)}

In this section we consider totally loxodromic representations of $\pi_1(S_g)$ to ${\rm SL}(3,{\mathbb R})$. 
We are interested in those 
representations that can be connected to the Fuchsian representations, that is those whose image lies in
${\rm SO}_0(2,1)$, through a continuous path of
representations. Choi and Goldman \cite{Choi-Gol} showed that the component of ${\rm SL}(3,{\mathbb R})$ 
representations containing the Fuchsian representations corresponds to the space of convex real projective 
structures on $S_g$. This component is called the Hitchin component. Since any loxodromic element of
of ${\rm SO}_0(2,1)$ has positive eigenvalues, each non-trivial element of $\pi_1(S_g)$ will be represented by
a loxodromic map with positive eigenvalues. In \cite{Gol}, Goldman defined Fenchel-Nielsen coordinates
for such representations. His parameters are boundary parameters, internal parameters and twist-bulge parameters.
These correspond to our trace parameters, shape invariants and twist-bulge parameters respectively. Goldman's internal
parameters were not symmetric under cyclic permutation of the boundary curves of each three-holed sphere, but later
Zhang \cite{Zhang2} showed how to symmetrise them. We will use Zhang's parameters. 

\subsection{Fenchel-Nielsen coordinates for ${\rm SL}(3,{\mathbb R})$}

It is clear that for representations to ${\rm SL}(3,{\mathbb R})$ the trace parameters
${\rm tr}(A_j^{\pm 1})$ and the shape invariants $\sigma_\pm(Y_k)$ are all real. Using Lawton's theorem,
we see that $\rho\bigl(\pi_1(Y_k)\bigr)$ is determined by these parameters together with a choice of
root of the commutator quadratic $Q(Y_k)$. 

Now consider the twist-bend-bulge-turn parameters. Recall from Section~\ref{sec:alg-FN} that if $A$ is loxodromic
then an element $K$ of the centraliser $Z(A)$ of $A=\rho([\alpha])$ can be written as $T^uU^v$ and has eigenvalues 
$e^{u-v}$, $e^{2v}$ and $e^{-u-v}$. For representations to ${\rm SL}(3,{\mathbb R})$ these all need to be real.
Hence the bend $\theta_\Gamma(\alpha)={\rm Im}(u)$ and the turn $\phi_\Gamma(\alpha)={\rm Im}(v)$ are zero. 
We are left with the twist and bulge parameters
$t_\Gamma(\alpha)={\rm Re}(u)$ and $s_\Gamma(\alpha)={\rm Re}(v)$, see Section~5.5 of Goldman \cite{Gol}. 
Note that in the next section we will use $s$ and $t$ to denote Goldman's internal parameters rather than the bulge and twist. \
The context will make this clear.

Thus we have proved that $\rho:\bigl(\pi_1(S_g)\bigr)\longrightarrow {\rm SL}(3,{\mathbb R})$ is determined by
\begin{enumerate}
\item[(1)] $6g-6$ real trace parameters ${\rm tr}(A_1),\,\ldots,\,{\rm tr}(A_{3g-3})$ and 
${\rm tr}(A_1^{-1}),\,\ldots,\,{\rm tr}(A_{3g-3}^{-1})$;
\item[(2)] $4g-4$ real shape invariants $\sigma_+(Y_1),\,\ldots,\,\sigma_+(Y_{2g-2})$ and 
$\sigma_-(Y_1),\,\ldots,\,\sigma_-(Y_{2g-2})$;
\item[(3)] a choice of root for each of the $2g-2$ polynomials $Q(Y_1),\,\ldots,\, Q(Y_{2g-2})$;
\item[(4)] $3g-3$ real twist parameters $t_\Gamma(\gamma_1),\,\ldots,\, t_\Gamma(\gamma_{3g-3})$ and 
$3g-3$ real bulge parameters $s_\Gamma(\gamma_1),\,\ldots,\, s_\Gamma(\gamma_{3g-3})$.
\end{enumerate}
This proves Theorem~\ref{thm:main-SL(3,R)}. 

Moreover, consider a subgroup $\langle A,B,C:CBA=I\rangle$ of ${\rm SL}(3,{\mathbb C})$ where 
${\rm tr}(A^{\pm 1})$, ${\rm tr}(B^{\pm 1})$, ${\rm tr}(C^{\pm 1})$, and $\sigma_\pm$ are all real. 
Using Lawton's theorem all traces in this group are determined by real polynomial functions of these traces, and so 
must themselves be real. Hence, using Acosta's theorem we see that the group is conjugate to a subgroup
of ${\rm SL}(3,{\mathbb R})$. This proves Theorem~\ref{thm-pants} (2).

\subsection{Goldman-Zhang coordinates}

Now we relate our method of parameterising loxodromic maps with Goldman's. This relates our trace parameters and
shape invariants to Goldman's boundary and internal parameters. (Our twist and bulge parameters agree with his.)

Suppose that $A\in{\rm SL}(3,{\mathbb R})$ is loxodromic with (real) eigenvalues $\lambda_A$, $\mu_A$, $\nu_A$ satisfying
$0<\lambda_A<\mu_A<\nu_A$. Note this implies $0<\lambda_A<1$. Goldman defines $\tau_A=\mu_A+\nu_A$ and he shows that 
$2/\sqrt{\lambda_A}<\tau_A<\lambda_A+\lambda_A^{-2}$. Since the eigenvalues of $A^{-1}$ are $\lambda_A^{-1}$, 
$\mu_A^{-1}=\lambda_A\nu_A$ and $\nu_A^{-1}=\lambda_A\mu_A$ we see
$$
{\rm tr}(A)=\lambda_A+\tau_A,\quad {\rm tr}(A^{-1})=\lambda_A^{-1}+\lambda_A\tau_A.
$$
It is easy to see that the Jacobian of the map $(\lambda_A,\tau_A)\longmapsto\bigl({\rm tr}(A),{\rm tr}(A^{-1})\bigr)$
is zero if and only if $\tau_A=\lambda_A+\lambda_A^{-2}$. Hence when $\tau_A<\lambda_A+\lambda_A^{-2}$ there is a 
bijection between the parameters $\bigl({\rm tr}(A),{\rm tr}(A^{-1})\bigr)$ and $(\lambda_A,\tau_A)$. The inverse map
can be constructed by solving the characteristic polynomial of $A$, whose coefficients are determined by
${\rm tr}(A)$ and ${\rm tr}(A^{-1})$.

Now consider a triple of loxodromic maps $A,B,C$ in ${\rm SL}(3,{\mathbb R})$ with positive eigenvalues and satisfying
$CBA=I$. Goldman paramerises this triple by $(\lambda_A,\tau_A,\lambda_B,\tau_B,\lambda_C,\tau_C)$, which he calls
boundary invariants, and two further parameters $s$ and $t$, called internal parameters. 
Goldman's internal parameter $s$ is invariant under cyclic permutation
of $A$, $B$, $C$, but $t$ is not. Zhang \cite{Zhang2} remedied this by defining a parameter $r$. We will relate our 
parameters $\sigma_\pm$ with Zhang's parameters $s$ and $r$. 

Let ${\bf r}_A$, ${\bf r}_B$ and ${\bf r}_C$ be vectors in ${\mathbb R}^3$ corresponding to the repelling fixed points 
of $A$, $B$ and $C$. The parameter $s$ may be expressed in terms of certain ${\rm SL}(3,{\mathbb R})$ invariant
cross-ratios denoted $(a,b,c,d)_e$ as follows, see Section 4 of \cite{Gol} or equation~2.2 of \cite{Zhang2}.
\begin{eqnarray*}
\rho_A(s) := \bigl(A^{-1}{\bf r}_B,{\bf r}_C,{\bf r}_B,A{\bf r}_C\bigr)_{{\bf r}_A}
& = & 1+\sqrt{\frac{\lambda_C\lambda_A}{\lambda_B}}\tau_As+\frac{\lambda_C}{\lambda_B}s^2, \\
\rho_B(s) := \bigl(B^{-1}{\bf r}_C,{\bf r}_A,{\bf r}_C,B{\bf r}_A\Bigr)_{{\bf r}_B}
& = & 1+\sqrt{\frac{\lambda_A\lambda_B}{\lambda_C}}\tau_Bs+\frac{\lambda_A}{\lambda_C}s^2, \\
\rho_C(s) := \bigl(C^{-1}{\bf r}_A,{\bf r}_B,{\bf r}_A,C{\bf r}_B\Bigr)_{{\bf r}_C}
& = & 1+\sqrt{\frac{\lambda_B\lambda_C}{\lambda_A}}\tau_As+\frac{\lambda_C}{\lambda_B}s^2, \\
\end{eqnarray*}
This defines the internal parameter $s$. Following Zhang, Proposition~2.19 of \cite{Zhang2}, we define
the internal parameter $r$ by 
\begin{eqnarray*}
r & = & \Bigl(\bigl(B^{-1}{\bf r}_C,{\bf r}_B,B{\bf r}_A,{\bf r}_C\bigr)_{{\bf r}_A}-1\Bigr)
\Bigl(C^{-1}{\bf r}_A,{\bf r}_B,{\bf r}_A,C{\bf r}_B\Bigr)_{{\bf r}_C} \\
& = & \Bigl(\bigl(A^{-1}{\bf r}_B,{\bf r}_A,A{\bf r}_C,{\bf r}_B\bigr)_{{\bf r}_C}-1\Bigr)
\Bigl(B^{-1}{\bf r}_C,{\bf r}_A,{\bf r}_C,B{\bf r}_A\Bigr)_{{\bf r}_B} \\
& = & \Bigl(\bigl(C^{-1}{\bf r}_A,{\bf r}_C,C{\bf r}_B,{\bf r}_A\bigr)_{{\bf r}_B}-1\Bigr)
\Bigl(A^{-1}{\bf r}_B,{\bf r}_C,{\bf r}_B,A{\bf r}_C\Bigr)_{{\bf r}_A}.
\end{eqnarray*}
Note that Goldman's internal parameter $t$ is given by $t=r/\rho_B(s)$.

In \cite{Gol} Goldman gives implicit matrices for the representation of the three boundary elements of a pair of pants 
$Y$, this matrices are
\begin{eqnarray*}
A & = & \left(\begin{matrix}
\alpha_1 & \alpha_1a_2+\gamma_1a_3c_2 & \gamma_1a_3 \\
0 & -\beta_1+\gamma_1b_3c_2 & \gamma_1b_3 \\
0 & -\gamma_1c_2 & -\gamma_1 \end{matrix}\right), \\
B & = & \left(\begin{matrix}
-\alpha_2 & 0 & -\alpha_2a_3 \\
\alpha_2b_1 & \beta_2 & \beta_2b_3+\alpha_2a_3b_1 \\
\alpha_2c_1 & 0 & -\gamma_2+\alpha_2a_3c_1 \end{matrix}\right),\\
C & = & \left(\begin{matrix}
-\alpha_3+\beta_3a_2b_1 & \beta_3a_2 & 0 \\
-\beta_3b_1 & -\beta_3 & 0 \\
\gamma_3c_1+\beta_3b_1c_2 & \beta_3c_2 & \gamma_3 \end{matrix}\right) 
\end{eqnarray*}

Where
$$ 
\begin{array}{lll}
\alpha_1\beta_1\gamma_1={\rm det}(A)=1,\quad
\lambda_A=\alpha_1,\quad \tau_A=-\beta_1+\gamma_1(b_3c_2-1), \\
\alpha_2\beta_2\gamma_2={\rm det}(B)=1,\quad
\lambda_B=\beta_2,\quad \tau_B=-\gamma_2+\alpha_2(a_3c_1-1), \\
\alpha_3\beta_3\gamma_3={\rm det}(C)=1,\quad
\lambda_C=\gamma_3,\quad \tau_C=-\alpha_3+\beta_3(a_2b_1-1).
\end{array}
$$
The inverses of $A$, $B$ and $C$ are given by
\begin{eqnarray*}
A^{-1} & = & \left(\begin{matrix}
\alpha_1^{-1} & \beta_1^{-1}a_2 & \alpha_1^{-1}a_3+\beta_1^{-1}a_2b_3 \\
0 & -\beta_1^{-1} & -\beta_1^{-1}b_3 \\
0 & \beta_1^{-1}c_2 & -\gamma_1^{-1}+\beta_1^{-1}b_3c_2 
\end{matrix}\right), \\
B^{-1} & = & \left(\begin{matrix}
-\alpha_2^{-1}+\gamma_2^{-1}a_3c_1 & 0 & \gamma_2^{-1}a_3 \\
\beta_2^{-1}b_1+\gamma_2^{-1}b_3c_1 & \beta_2^{-1} & \gamma_2^{-1}b_3 \\
-\gamma_2^{-1}c_1 & 0 & -\gamma_2^{-1}
\end{matrix}\right), \\
C^{-1} & = & \left(\begin{matrix}
-\alpha_3^{-1} & -\alpha_3^{-1}a_2 & 0 \\
\alpha_3^{-1}b_1 & -\beta_3^{-1}+\alpha_3^{-1}a_2b_1 & 0 \\
\alpha_3^{-1}c_1 & \gamma_3^{-1}c_2+\alpha_3^{-1}a_2c_1 & \gamma_3^{-1}
\end{matrix}\right).
\end{eqnarray*}

Since we have the presentation $CBA=I$ then 
$$
\alpha_1\alpha_2\alpha_3=\beta_1\beta_2\beta_3=\gamma_1\gamma_2\gamma_3=1.
$$  

Changing variables as Goldman does and using Zhang's symmetrised coordinates, we have 
$$
\begin{array}{lll}
\alpha_1=\lambda_A, \quad & 
\displaystyle{\alpha_2=\sqrt{\frac{\lambda_C}{\lambda_A\lambda_B}}\,\frac{1}{s},}\quad &
\displaystyle{\alpha_3=\sqrt{\frac{\lambda_B}{\lambda_A\lambda_C}}\,s, }\\
\displaystyle{\beta_1=\sqrt{\frac{\lambda_C}{\lambda_A\lambda_B}}\,s, }\quad & 
\beta_2=\lambda_B, \quad & 
\displaystyle{\beta_3=\sqrt{\frac{\lambda_A}{\lambda_B\lambda_C}}\,\frac{1}{s},}\\
\displaystyle{\gamma_1=\sqrt{\frac{\lambda_B}{\lambda_A\lambda_C}}\,\frac{1}{s},}\quad &
\displaystyle{\gamma_2=\sqrt{\frac{\lambda_A}{\lambda_B\lambda_C}}\,s,}\quad &
\gamma_3=\lambda_C
\end{array}
$$
and
$$
\begin{array}{lll}
\displaystyle{ a_2 = \frac{r}{\rho_B(s)},}\quad &
a_3=2,\quad &
\displaystyle{ b_1 =\frac{\rho_B(s)\rho_C(s)}{r},} \quad \\
b_3 = 2, \quad & 
\displaystyle{c_1=\frac{\rho_B(s)}{2}}, \quad &
\displaystyle{c_2=\frac{\rho_A(s)}{2}}.
\end{array}
$$
Then a lengthy, but straightforward calculation yields
\begin{eqnarray*}
\sigma_+ & = &  \left(\sqrt{\lambda_A\lambda_B\lambda_C}+\frac{1}{\sqrt{\lambda_A\lambda_B\lambda_C}}\right)s 
+\left(r+\frac{\rho_A(s)\rho_B(s)\rho_C(s)}{r}\right)\frac{1}{s^2} \\
&&\quad +\left(\sqrt{\frac{\lambda_C}{\lambda_B}}\,\sqrt{\lambda_A}\tau_A
+\sqrt{\frac{\lambda_A}{\lambda_C}}\,\sqrt{\lambda_B}\tau_B
+\sqrt{\frac{\lambda_B}{\lambda_A}}\,\sqrt{\lambda_C}\tau_C\right)\frac{1}{s}+\frac{2}{s^2}, \\
\sigma_- & = & \left(\sqrt{\lambda_A\lambda_B\lambda_C}+\frac{1}{\sqrt{\lambda_A\lambda_B\lambda_C}}\right)\frac{1}{s}
+\left(\sqrt{\lambda_A\lambda_B\lambda_C}\, r+\frac{\rho_A(s)\rho_B(s)\rho_C(s)}
{\sqrt{\lambda_A\lambda_B\lambda_C}\,r}\right)\frac{1}{s} \\
&& \quad +\left(\sqrt{\frac{\lambda_B}{\lambda_C}}\,\sqrt{\lambda_A}\tau_A 
+\sqrt{\frac{\lambda_C}{\lambda_A}}\,\sqrt{\lambda_B}\tau_B 
+\sqrt{\frac{\lambda_A}{\lambda_B}}\,\sqrt{\lambda_C}\tau_C\right) s + 2s^2. 
\end{eqnarray*}

\bigskip

\section{${\rm SU}(2,1)$-coordinates}\label{sec:SU(2,1)}

In \cite{PP} Parker and Platis constructed Fenchel Nielsen coordinates for surface groups. Much of the construction
we have given in previous sections is modelled on their coordinates. However, there is one big difference. Parker and Platis
did not give coordinates for $\pi_1(Y)$ that are invariant under cyclic permutation of the three boundary curves.
Instead they focussed on two of them, $\alpha$ and $\beta$, represented by $A$ and $B$ respectively. They then used
${\rm tr}(A)$, ${\rm tr}(B)$ and the cross-ratios of the fixed points of $A$ and $B$. In this section we will show
that there is a bijection between our coordinates and the Parker-Platis coordinates.

\subsection{Hermitian forms and ${\rm SU}(2,1)$}

Consider a Hermitian form $\langle \cdot, \cdot \rangle$ on ${\mathbb C}^3$. We can write this form in terms of a matrix
$J$ and we suppose this form has signature $(2,1)$. In what follows, we suppose $J$ is given by \eqref{eq:J}. The group
${\rm SU}(J)$ is the group of matrices with determinant 1 that preserve the form $\langle \cdot,\cdot\rangle$. From this
it follows that any $A$ in ${\rm SU}(J)$ satisfies $A^*JA=J$ where $A^*$ is the conjugate transpose matrix of $A$. 
That is, $A^{-1}=J^{-1}A^*J$. Since ${\rm tr}(A^*)=\overline{{\rm tr}(A)}$, we make
the important observation that ${\rm tr}(A^{-1})=\overline{{\rm tr}(A)}$. Applying Lawton's theorem, we immediately
see that $\rho:\pi_1(Y)\longrightarrow {\rm SU}(J)$ is determined up to conjugation by 
${\rm tr}(A)$, ${\rm tr}(B)$, ${\rm tr}(C)$ and $\sigma_+$ together with a root of the quadratic $Q(Y)$. Since the roots
of the latter are the traces of $[A,B]$ and its inverse, these roots are complex conjugates of each other, see
Parker \cite{Par}.

Suppose that $A\in{\rm SU}(J)$ is loxodromic, that is its eigenvalues $\lambda$, $\mu$, $\nu$ satisfy 
$|\lambda|>|\mu|>|\nu|$ and $\lambda\mu\nu=1$. Now the eigenvalues of $A^{-1}$ are the same as those
of $A^*$. The former are $\lambda^{-1}$, $\mu^{-1}$, $\nu^{-1}$ and those of the latter are 
$\overline{\lambda}$, $\overline{\mu}$ and $\overline{\nu}$. By looking at their absolute values, we immediately
see that $\lambda^{-1}=\overline{\nu}$, $\mu^{-1}=\overline{\mu}$ and $\nu^{-1}=\overline{\lambda}$. Thus,
$\nu=\overline{\lambda}^{-1}$ and $\mu=\lambda^{-1}\nu^{-1}=\lambda^{-1}\overline{\lambda}$. Hence the
trace of $A$ is completely determined by $\lambda$. Indeed, for loxodromic maps there is a bijection between 
${\rm tr}(A)$ and $\lambda$, see Lemma~4.1 of Parker and Platis \cite{PP}.

\subsection{Fenchel-Nielsen coordinates for ${\rm SU}(2,1)$}

It is clear that for representations to ${\rm SU}(2,1)$ the trace parameters satisfy
${\rm tr}(A_j^{-1})=\overline{{\rm tr}(A_j)}$ and the shape invariants satisfy $\sigma_-(Y_k)=\overline{\sigma_+(Y_k)}$. 
Using Lawton's theorem, we see that $\rho\bigl(\pi_1(Y_k)\bigr)$ is determined by these parameters together with a choice of
root of the commutator quadratic $Q(Y_k)$. 

Now consider the twist-bend-bulge-turn parameters. Suppose that $A\in{\rm SU}(2,1)$ is loxodromic
and $K$ is in the centraliser $Z(A)$ of $A=\rho([\alpha])$. Without loss of generality, we can write
$$
A=\left(\begin{matrix} \lambda & 0 & 0 \\ 0 & \lambda^{-1}\overline{\lambda} & 0 \\
0 & 0 & \overline{\lambda}^{-1} \end{matrix}\right),\quad
K=T^uU^v=\left(\begin{matrix} e^{u-v} & 0 & 0 \\ 0 & e^{2v} & 0 \\ 0 & 0 & e^{-u-v} \end{matrix}\right).
$$
The $e^{2v}$ eigenspace of $K$ is the same as the $\lambda^{-1}\overline{\lambda}$ eigenspace of $A$,
which is in $V_+$. Therefore $|e^{2v}|=1$ and so $v$ is purely imaginary. In particular, the bulge 
$s_\Gamma(\alpha)={\rm Re}(v)$ is zero.
Furthermore, we have $e^{-u-v}=\overline{e^{u-v}}^{-1}=e^{-\overline{u}+\overline{v}}=e^{-\overline{u}-v}$.
In turn, this implies that $u=\overline{u}$ and so $u$ is real. In particular the bend 
$\theta_\Gamma(\alpha)={\rm Im}(u)$ is zero. Therefore, 
we only have twist and turn parameters $t_\Gamma(\alpha)={\rm Re}(u)$ and $\phi_\Gamma(\alpha)={\rm Im}(v)$. 
Note that Parker and Platis used the word bend for what we are calling turn.

Thus we have proved that $\rho:\bigl(\pi_1(S_g)\bigr)\longrightarrow {\rm SU}(2,1)$ is determined by
\begin{enumerate}
\item[(1)] $3g-3$ complex trace parameters ${\rm tr}(A_1),\,\ldots,\,{\rm tr}(A_{3g-3})$;
\item[(2)] $2g-2$ complex shape invariants $\sigma_+(Y_1),\,\ldots,\,\sigma_+(Y_{2g-2})$;
\item[(3)] a choice of root for each of the $2g-2$ polynomials $Q(Y_1),\,\ldots,\, Q(Y_{2g-2})$;
\item[(4)] $3g-3$ real twist parameters $t_\Gamma(\gamma_1),\,\ldots,\, t_\Gamma(\gamma_{3g-3})$ and 
$3g-3$ real turn parameters $\phi_\Gamma(\gamma_1),\,\ldots,\, \phi_\Gamma(\gamma_{3g-3})$.
\end{enumerate}
This proves Theorem~\ref{thm:main-SU(2,1)}. 

Moreover, consider a subgroup $\langle A,B,C:CBA=I\rangle$ of ${\rm SL}(3,{\mathbb C})$ where 
${\rm tr}(A^{-1})=\overline{{\rm tr}(A)}$, ${\rm tr}(B^{-1})=\overline{{\rm tr}(B)}$, ${\rm tr}(C^{-1})=\overline{{\rm tr}(C)}$
and $\sigma_-=\overline{\sigma_+}$. 
Using Lawton's theorem all traces in this group are determined by functions of these traces satisfy
${\rm tr}(W^{-1})=\overline{{\rm tr}(A)}$. Hence, using Acosta's theorem we see that the group is conjugate to a subgroup
of ${\rm SU}(2,1)$. This proves Theorem~\ref{thm-pants} (3).

\subsection{Parker-Platis coordinates}

It remains to discuss the relationship between our method of parametrisation of 
$\rho\bigl(\pi_1(Y)\bigr)=\langle A,B,C\,:\,CBA=I\rangle$ and that of Parker and Platis. 

The main difference between our parameterisation and that of Parker and Platis is that they use cross-ratios. 
Suppose that ${\bf r}_A$, ${\bf a}_A$ be repulsive and attractive eigenvectors of $A$ and 
${\bf r}_B$, ${\bf a}_B$ be repulsive and attractive eigenvectors of $B$ respetively. 
Following Section~6.1 of Parker and Platis \cite{PP} we define three cross-ratios associated to $A$ and $B$ as follows
\begin{eqnarray}
{\mathbb X}_1 & = & \frac{\langle{\bf r}_A,{\bf a}_B\rangle\langle{\bf r}_B,{\bf a}_A\rangle}
{\langle{\bf r}_B,{\bf a}_B\rangle\langle{\bf r}_A,{\bf a}_A\rangle}, \label{eq:X1} \\
{\mathbb X}_2 & = & \frac{\langle{\bf a}_A,{\bf a}_B\rangle\langle{\bf r}_B,{\bf r}_A\rangle}
{\langle{\bf r}_B,{\bf a}_B\rangle\langle{\bf a}_A,{\bf r}_A\rangle}, \label{eq:X2} \\
{\mathbb X}_3 & = & \frac{\langle{\bf a}_B,{\bf a}_A\rangle\langle{\bf r}_B,{\bf r}_A\rangle}
{\langle{\bf r}_B,{\bf a}_A\rangle\langle{\bf a}_B,{\bf r}_A\rangle}. \label{eq:X3}
\end{eqnarray}
Falbel \cite{Falbel} showed they satisfy the following equations, see also Proposition~5.2 of \cite{PP}
\begin{eqnarray*}
|{\mathbb X}_2| & = & |{\mathbb X}_1|\,|{\mathbb X}_3|, \\
2|{\mathbb X}_1|^2{\rm Re}({\mathbb X}_3) 
& = & |{\mathbb X}_1|^2+|{\mathbb X}_2|^2+1-2{\rm Re}({\mathbb X}_1+{\mathbb X}_2).
\end{eqnarray*}
Note that these two equations determine $|{\mathbb X}_3|$ and ${\rm Re}({\mathbb X}_3)$ in terms of
${\mathbb X}_1$ and ${\mathbb X}_2$. But there remains an ambiguity in the choice of sign of
${\rm Im}({\mathbb X}_3)$.

In \cite{PP} Parker and Platis use $(\lambda_A,\lambda_B,{\mathbb X}_1,{\mathbb X}_2,{\mathbb X}_3)$ to 
parametrise $\rho\bigl(\pi_1(Y)\bigr)=\langle A,B\rangle$.

\begin{proposition}
Suppose that $A$ and $B$ are loxodromic elements of ${\rm SU}(J)$. 
Let ${\mathbb X}_j$ for $j=1,2,3$ be the
cross-ratios of their eigenvectors as defined by \eqref{eq:X1}, \eqref{eq:X2}, \eqref{eq:X3}. Write $C=(AB)^{-1}$ and 
$\sigma_+={\rm tr}(A^{-1}B)-{\rm tr}(A^{-1}){\rm tr}(B)$. 
There is a bijection depending only on the eigenvalues of $A$ and $B$ between 
$(\lambda_A,\lambda_B,{\mathbb X}_1,{\mathbb X}_2)$ and $({\rm tr}(A),{\rm tr}(B),{\rm tr}(C),\sigma_+)$. 
Moreover, fixing the other parameters,
the sign of the imaginary part of ${\mathbb X}_3$ is determined by the choice of a root of the commutator quadratic $Q$.
\end{proposition}

\begin{proof}
Write the eigenvalues of $A$ and $B$ as 
$\lambda_A$, $\mu_A$ and $\nu_A$ and $\lambda_B$, $\mu_B$ and $\nu_B$ with $|\lambda_A|>|\mu_A|>|\nu_A|$ and 
$|\lambda_B|>|\mu_B|>|\nu_B|$. 

First, the eigenvalues of $A$ are the roots of the characteristic polynomial 
$$
\chi_A(x)=x^3-{\rm tr}(A)x^2+\overline{{\rm tr}(A)}x-1.
$$
Thus there is a bijection between ${\rm tr}(A)$ and the ordered set eigenvalues of $A$. 
Now suppose that $A\in{\rm SU}(J)$ is loxodromic. Since we have
$\mu_A=\lambda_A^{-1}\overline{\lambda}_A$ and $\nu_A=\overline{\lambda}_A^{-1}$, we see that there
is a bijection between the set of possible values of ${\rm tr}(A)$ and the set of possible values of 
$\lambda_A$, see Lemma~4.1 of Parker and Platis.

We know that ${\rm tr}(C)={\rm tr}(A^{-1}B^{-1})$ and ${\rm tr}(C^{-1})={\rm tr}(BA)$. Also
$$
\sigma_-={\rm tr}(B^{-1}A)-{\rm tr}(B^{-1}){\rm tr}(A)=\overline{\sigma}_+.
$$ 
Therefore it suffices to show there
is a bijection between the two sets 
$({\mathbb X}_1,\overline{\mathbb X}_1,{\mathbb X}_2,\overline{\mathbb X}_2)$ and
$({\rm tr}(BA),{\rm tr}(A^{-1}B^{-1}),{\rm tr}(A^{-1}B),{\rm tr}(AB^{-1}))$. 
 
As above, write the eigenvalues of $A$ and $B$ as 
$\lambda_A$, $\mu_A$ and $\nu_A$ and $\lambda_B$, $\mu_B$ and $\nu_B$ with $|\mu_A|=|\mu_B|=1$.
Then from Proposition 7.6 of Parker-Platis \cite{PP} we have
\begin{eqnarray*}
\lefteqn{{\rm tr}(BA)-(\lambda_A+\nu_A)\mu_B-\mu_A(\lambda_B+\nu_B)+\mu_A\mu_B} \\
& = & (\nu_A-\mu_A)(\nu_B-\mu_B){\mathbb X}_1 
+(\lambda_A-\mu_A)(\lambda_B-\mu_B)\overline{\mathbb X}_1 \\
&& \quad +(\lambda_A-\mu_A)(\nu_B-\mu_B){\mathbb X}_2
+(\nu_A-\mu_A)(\lambda_B-\mu_B)\overline{\mathbb X}_2, \\
\lefteqn{{\rm tr}(A^{-1}B^{-1})-(\lambda_A^{-1}+\nu_A^{-1})\mu_B^{-1}
-\mu_A^{-1}(\lambda_B^{-1}+\nu_B^{-1})+\mu_A^{-1}\mu_B} \\
& = & (\nu_A^{-1}-\mu_A^{-1})(\nu_B^{-1}-\mu_B^{-1}){\mathbb X}_1 
+(\lambda_A^{-1}-\mu_A^{-1})(\lambda_B^{-1}-\mu_B^{-1})\overline{\mathbb X}_1 \\
&& \quad +(\lambda_A^{-1}-\mu_A^{-1})(\nu_B^{-1}-\mu_B^{-1}){\mathbb X}_2
+(\nu_A^{-1}-\mu_A^{-1})(\lambda_B^{-1}-\mu_B^{-1})\overline{\mathbb X}_2, \\
\lefteqn{{\rm tr}(A^{-1}B)-(\lambda_A^{-1}+\nu_A^{-1})\mu_B-\mu_A^{-1}(\lambda_B+\nu_B)+\mu_A^{-1}\mu_B} \\
& = & (\nu_A^{-1}-\mu_A^{-1})(\nu_B-\mu_B){\mathbb X}_1 
+(\lambda_A^{-1}-\mu_A^{-1})(\lambda_B-\mu_B)\overline{\mathbb X}_1 \\
&& \quad +(\lambda_A^{-1}-\mu_A^{-1})(\nu_B-\mu_B){\mathbb X}_2
+(\nu_A^{-1}-\mu_A^{-1})(\lambda_B-\mu_B)\overline{\mathbb X}_2, \\
\lefteqn{{\rm tr}(B^{-1}A)-(\lambda_A+\nu_A)\mu_B^{-1}-\mu_A(\lambda_B^{-1}+\nu_B^{-1})+\mu_A\mu_B^{-1}} \\
& = & (\nu_A-\mu_A)(\nu_B^{-1}-\mu_B^{-1}){\mathbb X}_1 
+(\lambda_A-\mu_A)(\lambda_B^{-1}-\mu_B^{-1})\overline{\mathbb X}_1 \\
&& \quad +(\lambda_A-\mu_A)(\nu_B^{-1}-\mu_B^{-1}){\mathbb X}_2
+(\nu_A-\mu_A)(\lambda_B^{-1}-\mu_B^{-1})\overline{\mathbb X}_2.
\end{eqnarray*}
This forms a set of linear equations in ${\mathbb X}_1$, $\overline{\mathbb X}_1$, 
${\mathbb X}_2$ and $\overline{\mathbb X}_2$. We can solve for the cross-ratios provided the determinant of
the corresponding matrix is non-zero. A brief calculation shows that this determinant is
\begin{eqnarray*}
\Delta & = & \Bigl( (\lambda_A-\mu_A)(\nu_A^{-1}-\mu_A^{-1})-(\nu_A-\mu_A)(\lambda^{-1}_A-\mu^{-1}_A)\Bigr)^2 \\
&& \quad \cdot \Bigl( (\lambda_B-\mu_B)(\nu_B^{-1}-\mu_B^{-1})-(\nu_B-\mu_B)(\lambda^{-1}_B-\mu^{-1}_B)\Bigr)^2 \\
& = & (\lambda_A-\nu_A)^2(\lambda_A-\mu_A)^2(\nu_A-\mu_A)^2
(\lambda_B-\nu_B)^2(\lambda_B-\mu_B)^2(\nu_B-\mu_B)^2.
\end{eqnarray*}
On the last line we used $\lambda_A\mu_A\nu_A=\lambda_B\mu_B\nu_B=1$. Since $A$ and $B$ were assumed
to be loxodromic we see they do not have repeated eigenvalues, and hence $\Delta\neq 0$.

Furthermore, given $(\lambda_A,\lambda_B,{\mathbb X}_1,{\mathbb X}_2)$, or equivalently
$({\rm tr}(A),{\rm tr}(B),{\rm tr}(C),\sigma_+)$, using Corollary 6.5 of \cite{PP}, we have
$$
{\mathbb X}_3 = \frac{F(\lambda_A,\lambda_B,{\mathbb X}_1,{\mathbb X}_2)-{\rm tr}[A,B]}
{|{\mathbb X}_1|^2|\lambda_A|^2|\nu_A-\mu_A|^2|\lambda_A|^2|\nu_B-\mu_B|^2}
$$
where $F(\lambda_A,\lambda_B,{\mathbb X}_1,{\mathbb X}_2)$ is a real valued, real analytic function.
Thus, the ambiguity in the roots of the commutator equation is the same as the
ambiguity in the sign of the imaginary part of ${\mathbb X}_3$. This completes the proof.
\end{proof}

\end{document}